\documentclass[10pt]{amsart}
\usepackage{egastyle}

\usepackage{amssymb,amsmath,amscd,amsthm,amsfonts,wasysym,mathrsfs,amsxtra,stmaryrd,array,ulem}

\usepackage{graphicx,array,url}
\usepackage[vcentermath]{genyoungtabtikz}
\usepackage{tikz}
\usepackage[section]{placeins}
\usepackage{afterpage}
\usepackage{verbatim}
\usepackage{tikz-cd}

\input xy
\xyoption{all}

\newcommand{\OO}{\mathcal{O}}

\newcommand{\kernel}{\textnormal{ker}\,}
\newcommand{\image}{\textnormal{im}}
\newcommand{\cokernel}{\textnormal{coker}\,}

\newcommand{\Hom}{\textnormal{Hom}}
\newcommand{\dimension}{\textnormal{dim}\,}

\newcommand{\Ext}{\textnormal{Ext}}

\newcommand{\Ac}{\mathcal{A}}
\newcommand{\Bc}{\mathcal{B}}

\newcommand{\Fc}{\mathcal{F}}
\newcommand{\Tc}{\mathcal{T}}

\newcommand{\Hc}{\mathcal{H}}

\newcommand{\Coh}{\mathrm{Coh}}

\newcommand{\arinj}{\ar@{^{(}->}}
\newcommand{\arsurj}{\ar@{->>}}
\newcommand{\areq}{\ar@{=}}

\newcommand{\Bl}{\mathcal{B}^l}

\newcommand{\ch}{\mathrm{ch}}

\newcommand{\bsm}{\begin{smallmatrix}}
\newcommand{\esm}{\end{smallmatrix}}

\newcommand{\A}{\mathcal{A}}
\newcommand{\B}{\mathcal{B}}
\newcommand{\T}{\mathcal{T}}
\newcommand{\F}{\mathcal{F}}
\newcommand{\K}{\mathcal{K}}

\newcommand{\cP}{\mathcal{P}}

\newcommand{\cO}{\mathcal{O}}

\newcommand{\cblue}[1]{{\textcolor{black}{#1}}} 

\makeatletter
\newtheorem*{rep@theorem}{\rep@title}
\newcommand{\newreptheorem}[2]{%
\newenvironment{rep#1}[1]{%
 \def\rep@title{#2 \ref{##1}}%
 \begin{rep@theorem}}%
 {\end{rep@theorem}}}
\makeatother

  {\begin{genthm}{Question}{true}{#1}{paragraph}}%
  {\end{genthm}}

  {\begin{genthm}{Question}{true}{#1}{equation}}%
  {\end{genthm}}

\newreptheorem{theorem}{Theorem}

\usepackage{scalerel,stackengine}
\stackMath
\newcommand\reallywidehat[1]{%
\savestack{\tmpbox}{\stretchto{%
  \scaleto{%
    \scalerel*[\widthof{\ensuremath{#1}}]{\kern-.6pt\bigwedge\kern-.6pt}%
    {\rule[-\textheight/2]{1ex}{\textheight}}
  }{\textheight}%
}{0.5ex}}%
\stackon[1pt]{#1}{\tmpbox}%
}
\begin{document}

\title[]{Weak stability conditions as limits of Bridgeland stability conditions}

\author[Tristan C. Collins]{Tristan C. Collins}
\address{Department of Mathematics \\
Massachusetts Institute of Technology\\
77 Massachusetts Avenue\\
Cambridge MA 02139 \\
USA}
\email{tristanc@mit.edu}

\author[Jason Lo]{Jason Lo}
\address{Department of Mathematics \\
California State University, Northridge\\
18111 Nordhoff Street\\
Northridge CA 91330 \\
USA}
\email{jason.lo@csun.edu}

\author[Yun Shi]{Yun Shi}
\address{Department of Mathematics \\
Brandeis University\\
415 South Street\\
Waltham MA 02453\\
USA}
\email{yunshi@brandeis.edu}

\author[Shing-Tung Yau]{Shing-Tung Yau}
\address{Yau Mathematical Sciences Center\\
Tsinghua University\\
Haidian District, Beijing\\
China}
\email{styau@tsinghua.edu.cn}


\begin{abstract}
In this paper, we give a definition of weak stability condition on a triangulated category.  The difference between our definition and existing definitions is that we allow objects in the kernel to have non-maximal phases.  We then construct four types of weak stability conditions that naturally occur on Weierstra{\ss} ellitpic surfaces as limites of Bridgeland stability conditions. 
\end{abstract}

\maketitle
\tableofcontents

\section{Introduction}
The notion of weak stability conditions was defined in \cite{piyaratne2015moduli} and has been studied by many authors (e.g.\  see  \cite{piyaratne2015moduli}, \cite{BMT1}, \cite{broomhead2022partial}). Recall that in \cite{StabTC}, a Bridgeland stability condition on the derived category of coherent sheaves of a smooth projective variety can be equivalently defined as a pair $(Z, \A)$, where $Z$ is a group homomorphism from the Grothendieck group of the derived category to the complex numbers, called the central charge, and $\A$ is the heart of a bounded t-structure. Weak stability conditions are defined analogously, the main difference being weak stability conditions allow nonzero objects in the heart $\A$ to lie in the kernel of the central charge $Z$, i.e. there can be objects $0 \neq E\in\A$ such that $Z(E)=0$.

In this paper, we generalize the definition of weak stability conditions as defined in \cite{piyaratne2015moduli}. This is motivated by the observation that weak stability conditions also naturally occur as degenerations - or `limits' - of Bridgeland stability conditions. When the data $(Z, \A)$ of  a Bridgeland stability condition approaches a certain limit, it can happen that  $\kernel(Z)$ begins to have nonzero intersection with $\A$. 
One obvious issue is how to define the phases of the objects in $\ker(Z) \cap \A$. In the original definition of weak stability conditions in \cite{piyaratne2015moduli}, the phase of any object in $\ker(Z) \cap \A$ is always set to  be maximal. This is consistent with  intuition, since $\ker(Z) \cap \A$ of weak stability conditions considered in \cite{piyaratne2015moduli} merely contains  skyscraper sheaves.
But if we consider weak stability conditions that occur as limits of Bridgeland stability conditions in general, it can happen that $\ker(Z)\cap \A$ contains objects not supported in dimension zero. 
In this paper, we propose a natural generalization - we define the phases of the objects in $\ker(Z) \cap \A$ to be the limits of their Bridgeland phases as the data $(Z, \A)$ of a Bridgeland stability condition approaches a limit.  

The first application of this generalized definition is in our followup paper  \cite{CLSY}, where we study stable objects with respect to Bridgeland stability conditions. In some cases, it is easier to show an object is stable with respect to weak stability conditions rather than  Bridgeland stability conditions. Our new definition then allows us to use the stability of an object at a particular weak stability condition - considered as a limit of Bridgeland stability conditions -  to conclude the stability of the object with respect to nearby Bridgeland stability conditions. 

Besides the above application, we also 
hope that this generalization can give a better understanding of the Bridgeland stability manifold. In particular, we hope that the weak stability conditions defined in this paper help give a better picture of the boundaries of certain open components of the Bridgeland stability manifold. 
We note that this direction has already been undertaken  in \cite{broomhead2022partial}, where the authors also take an approach towards weak stability conditions using slicings.  Our definition of a weak stability condition, however, does not require the slicing to be locally finite, nor require the existence of a support property.  We do, however, require the weak see-saw property.  The main motivation behind our definition is that we imagine a weak stability condition to be a degeneration of Bridgeland stability conditions and, in our examples, we always define the phases of objects in the kernel to be limits of their phases with respect to the associated Bridgeland stability conditions.  Our  main application in this article is in proving the Bridgeland stability of specific objects, while that of \cite{broomhead2022partial} is in compactifying the Bridgeland stability manifold. 

\subsection{Outline}

In Section \ref{sec:prelim}, we  set up some notation for elliptic surfaces.  In Section \ref{sec:def-wsc}, we recall the definition of very weak stability conditions as defined in \cite{piyaratne2015moduli} and give a generalized definition of weak stability conditions. We also discuss  several  properties of weak stability conditions that are analogous to those of Bridgeland stability conditions, e.g. slicing, Harder-Narasimhan property etc.   In Section \ref{sec:BGi}, we show that a stronger version of  Bogomolov-Gieseker inequality for slope stable sheaves holds for some nef diviors, too.

Then, in Section \ref{sec:fourtypesofweakstab}, we construct four types of weak stability conditions that satisfy our definition. These are the weak stability conditions we encounter in our followup paper, but they are of their own interests. In particular, they include the following types of limits of Bridgeland stability conditions: (i) when the ample class approaches a nef class; (ii) when the coefficient of $\ch_0$ in the central charge $Z$ of a Bridgeland stability condition approaches zero; (iii) when the limits of (i) and (ii) happen at the same time; and (iv) the image of the weak stability condition in (iii) under the relative Fourier Mukai transform. In each of these cases, it is easy to write down what the central charge of the weak stability condition ``should'' be; it is less clear, however, as to how to construct the right heart of t-structure so that both the positivity property and the Harder-Narasimhan property are satisfied.  In addition, in order to construct a weak stability condition using our definition, we need to describe the objects in the intersection of the heart and the kernel of the central charge.  Most of Section \ref{sec:fourtypesofweakstab} are devoted to answering these technical questions.

\subsection{Acknowledgements} JL was partially supported by NSF grant DMS-2100906.  TCC was partially supported by NSF CAREER grant DMS-1944952.  Part of this work was done when YS was a postdoc at CMSA, Harvard, she would like to thank CMSA for the excellent working environment.

\section{Preliminaries}\label{sec:prelim}

\paragraph[Notation] Let $X$ be a smooth projective variety $X$.  We will write $D^b(X)$ to denote $D^b(\Coh (X))$, the bounded derived category of coherent sheaves on $X$. 

In Section \ref{sec:def-wsc}, we will give a definition of weak stability conditions on any triangulated category.  In  later sections, however, we will focus on weak stability conditions on the derived category of coherent sheaves on a Weierstra{\ss} elliptic surface, and so we also review the definition of such surfaces here. 

\paragraph[Weierstra{\ss} elliptic surface] \label{para:def-Wellsurf}
By a Weierstra{\ss} elliptic surface, we mean a flat morphism  $p : X\to Y$ of smooth projective varieties of relative dimension 1, where $X$ is a surface and
\begin{itemize}
    \item the fibers of $p$ are Gorenstein curves of arithmetic genus 1, and are geometrically integral;
    \item $p$ has a section $s : Y \to X$ such that its image $\Theta$ does not intersect any singular point of any singular fiber of $p$.
\end{itemize}
The definition we use here follows that of 
 \cite[Definition 6.10]{FMNT}.  Under our definition, the generic fiber of $p$ is a smooth elliptic curve, and the singular fibers of $p$ are either nodal or cuspidal curves.  We usually write $f$ to denote the class of a fiber for the fibration $p$, and write $e = -\Theta^2$.  Often, we simply refer to $X$ as a Weierstra{\ss} elliptic surface.  Note that when $Y=\mathbb{P}^1$, $X$  is K3 if and only if $e=2$ \cite[2.3]{LLM}.

\paragraph[RDV coordinates for divisors] \label{para:RDVcoord} Suppose $X$ is a Weierstra{\ss} elliptic surface.  Given any divisor $M$ on $X$ of the form $M = a\Theta + bf$ where $a, b \in \mathbb{R}$ and $a \neq 0$, we can find real numbers $R_M \neq 0$ and $D_M$ such that
\begin{equation}\label{eq:RDVcoord-2}
  M = R_M (\Theta + (D_M + e)f).
\end{equation}
We also set
\[
  V_M = \frac{M^2}{2}=R_M^2(D_M+\tfrac{e}{2}).
\]
Note that when $D_M, V_M>0$ (e.g.\ when $M$ is ample), we can write
\[
  R_M = \sqrt{\frac{V_M}{D_M+\tfrac{e}{2}}}.  
\]

The coordinates $R_M, D_M, V_M$ for divisors $M$ are especially suited for computations on elliptic fibrations, and are inspired by symmetries first observed in \cite{LMpc}. For example, we have 
\begin{equation}\label{eq:RDVcoord-1}
  \Theta M=R_MD_M
\end{equation}
and, if $W$ is a divisor  written in the  form  \eqref{eq:RDVcoord-2}, then 
\[
 MW = R_M R_W (\Theta + (D_M+D_W + e)f)
\]
which is reminiscent of multiplication for complex numbers in polar coordinates.

\section{Definition of weak stability condition}\label{sec:def-wsc}


 We begin by recalling the following definition of a very weak pre-stability condition in \cite{piyaratne2015moduli}.  Let us denote the strict upper-half complex plane by $\mathbb{H} = \{ r e^{i\pi\phi} : r \in \mathbb{R}_{>0}, \phi \in (0,1) \}$.
	\begin{defn} (Definition 2.1 in \cite{piyaratne2015moduli})
		\label{Def:WeakPT16}
		A very weak pre-stability condition on a triangulated category $\mathcal{D}$ is a pair $(Z, \A)$, where $\A$ is the heart of a bounded t-structure on $\mathcal{D}$, and $Z: K(\mathcal{D})\rightarrow \mathbb{C}$ is a group homomorphism satisfying the following conditions:
		
		(i) For any $E\in \mathcal{D}$, we have $Z(E)\in \mathbb{H}\cup \mathbb{R}_{\leq 0}$.	
		
		(ii) Let $\rho=-\frac{\Re(Z)}{\Im(Z)}$ be the associated slope function, where we set $\rho(E)=\infty$ if $\Im Z(E)=0$. Then $\rho$ satisfies the Harder-Narasimhan (HN for short) property.
	\end{defn}
	
	In the next section, we will study weak stability conditions whose central charges are limits of central charges associated to Bridgeland stability conditions. Intuitively, we want the phases of the objects with respect to the limit stability conditions to be the limits of the phases with respect to the Bridgeland stability conditions. To achieve this goal, we modify  Definition \ref{Def:WeakPT16}. In particular, we allow the objects in the kernel of the central charge to have phases other than $1$, and define the phases of the objects in the kernel of the central charge to be the limits of the phases with respect to Bridgeland stability conditions.

	
	\begin{defn}
		\label{Def:Weak}
		A weak stability condition on $\mathcal{D}$ is a triple $$\sigma=(Z, \A, \{\phi(K)\}_{K\in \ker(Z)\cap \A}),$$ where $\mathcal{A}$ is the heart of a bounded t-structure on $\mathcal{D}$, and $Z: K(\mathcal{D})\rightarrow \mathbb{C}$ a group homomorphism satisfying:
		
		(i) For any $E\in \mathcal{A}$, we have $Z(E)\in\mathbb{H}\cup\mathbb{R}_{\leq 0}$. For any $K\in \ker(Z)\cap\A$, we have $0<\phi(K)\leq 1$. 
  
		(ii) (Weak see-saw property) For any short exact sequence 
		\begin{equation*}
			0\rightarrow K_1\rightarrow K\rightarrow K_2\rightarrow 0
		\end{equation*}
		in $\ker(Z)\cap\A$, we have $\phi(K_1)\geq \phi(K)\geq \phi(K_2)$ or $\phi(K_1)\leq \phi(K)\leq \phi(K_2)$.

		 For any object $E\notin \ker(Z)$, define the phase of an object $E\in\A$ by $$\phi(E)=(1/\pi) \mathrm{arg}\, Z(E)\in (0, 1].$$ We further require 
   
   (iii) The phase function satisfies the Harder-Narasimhan (HN) property. 
	\end{defn}

Note that to talk about HN property in (iii) of the above definition, one will need the notion of the slope of a stability condition and the notion of semistability, which we define below.
 

\subparagraph	\label{para:ssdefforweakstabcon}   Define the slope of an object with respect to a weak stability condition by 
$\rho(E)=-\text{cot}(\pi\phi(E))$
Following \cite{piyaratne2015moduli}, an object $E\in \A$ is (semi)-stable if for any nonzero subobject $F\subset E$ in $\A$, we have 
\[
\rho(F)<(\leq)\rho(E/F),
\]
or equivalently,
\[
\phi(F)<(\leq)\phi(E/F).
\]
	\begin{rem}\label{rem:link1}
         Note that the condition (ii) in Definition \ref{Def:Weak} (weak see-saw property) implies that for any short exact sequence in $\A$
		\begin{equation*}
			0\rightarrow F_1\rightarrow F\rightarrow F_2\rightarrow 0,
		\end{equation*}
		we have $\phi(F_1)\geq \phi(F_2)$ implies that $\phi(F_1)\geq\phi(F)\geq \phi(F_2)$, and $\phi(F_1)\leq \phi(F_2)$ implies that $\phi(F_1)\leq\phi(F)\leq \phi(F_2)$. 
	\end{rem}

The weak seesaw property built into the definition of a weak stability condition also gives the following property, which comes naturally with Bridgeland stability conditions:
 \begin{lemsub}\label{lem:AG52-37-1}
Let $(Z, \Ac, \{ \phi(K) \}_{K \in \kernel (Z) \cap \Ac})$ be a weak stability condition.  Then for any short exact sequence in $\Ac$
\[
0 \to E' \to E \to E'' \to 0
\]
such that $\phi(E'), \phi (E'')>\phi_0$ for some constant $\phi_0$, we also have $\phi (E)>\phi_0$.
 \end{lemsub}

 \begin{proof}
If at least one of $Z(E'), Z(E'')$ is nonzero, then the claim is clear.  Suppose $Z(E'), Z(E'')$ are both zero.  Then by the weak seesaw property, we have either $\phi (E') \geq \phi (E) \geq \phi (E'')$ or $\phi (E') \leq \phi (E) \leq \phi (E'')$.  It follows that 
\[
\phi (E) \geq \min{\{\phi (E'), \phi (E'')\}}> \phi_0
\]
and the claim follows.
 \end{proof}
	
	Let $(Z, \A, \{\phi(K)\}_{K\in \ker(Z)\cap\A})$ be a weak stability condition. Similar to Bridgeland stability conditions, for $0<\phi\leq 1$, we define the full additive subcategories $\cP(\phi)$ of $\mathcal{D}$ consisting of the objects in $\A$ which are semistable of phase $\phi$ together with the zero object in $\mathcal{D}$. For general $\phi$ define $\cP(\phi+1)=\cP(\phi)[1]$.
	\begin{prop}
		\label{Prop:Slic}
		$\cP$ defines a slicing in the sense of Definition 3.3 in \cite{StabTC}.
	\end{prop}
	\begin{proof}
		We only need to check the following axiom in Definition 3.3:
		\[
		\text{If $\phi_1>\phi_2$, $A_j\in\cP(\phi_j)$, then $\Hom_\mathcal{D}(A_1, A_2)=0$.}
		\]
		Without loss of generality, we assume that $A_2\in \A$. Since $\A$ is the heart of a bounded t-structure, if $\phi_1>1$ we have $\Hom_\mathcal{D}(A_1, A_2)=0$. So we may assume that $A_i\in \A$ for both $i$.
		Assuming there is a nonzero map $f: A_1\rightarrow A_2$,  we have the following short exact sequence in $\A$:
		$$0\rightarrow \kernel (f)\rightarrow A_1\rightarrow \image (f)\rightarrow 0,$$ 
		$$0\rightarrow \image (f)\rightarrow A_2\rightarrow \cokernel (f)\rightarrow 0.$$ 
		Since $A_j$'s are semistable, we have $\phi(\kernel (f))\leq \phi(\image (f))\leq \phi(\cokernel (f))$. 
  Then by the weak see-saw property, we have 
  \begin{equation*}
  \phi(\kernel (f))\leq \phi_1\leq \phi(\image (f))\leq\phi_2\leq \phi(\cokernel (f))
  \end{equation*}
  which contradict with $\phi_1>\phi_2$.
	\end{proof}

Let $\cP((0, 1])$ be the extension closure of all the $\cP(\phi)$ where $\phi \in (0,1]$.
\begin{prop}
A weak stability condition can be equivalently defined by 
$$\sigma=(Z, \cP, \{\phi(E)\}_{E\in \ker(Z)\cap\cP((0, 1])}),$$
where $\cP$ is a slicing with properties as in Definition 3.3 in \cite{StabTC}, and the central charge $Z: K(\mathcal{D})\rightarrow \mathbb{C}$ is a group homomorphism satisfying:

(i) For $E\in \cP(\phi)$,

\begin{itemize}
    \item if $E\notin \cP(\phi)\cap \ker(Z)$, then $Z(E)=m(E)e^{i\pi\phi}$ where $m(E)\in \mathbb{R}_{> 0}$;
    \item if $E\in \cP(\phi)\cap \ker(Z)$, then $\phi(E)=\phi$.
\end{itemize}

(ii) (weak see-saw property) For any short exact sequence 
		\begin{equation*}
			0\rightarrow K_1\rightarrow K\rightarrow K_2\rightarrow 0
		\end{equation*}
		in $\ker(Z)\cap\cP((0, 1])$, we have $\phi(K_1)\geq \phi(K)\geq \phi(K_2)$ or $\phi(K_1)\leq \phi(K)\leq \phi(K_2)$.
\end{prop}
\begin{proof}
Suppose we have a weak stability condition $(Z,\Ac, \{\phi(K)\}_{K\in \ker(Z)\cap \A})$ as in Definition \ref{Def:Weak}.    By Proposition \ref{Prop:Slic}, we have that the triple defines a slicing $\cP$, where $\A=\cP((0, 1])$. Then the phase of the objects in $\ker(Z)\cap\A$ in Definition \ref{Def:Weak} defines the phases $\{\phi(E) : E\in \ker(Z)\cap \cP((0, 1])\}$ and the weak see-saw property implies (ii) in the Proposition.


Now assume that we have a triple $(Z, \cP, \phi(E)|_{E\in \ker(Z)\cap\cP((0, 1])})$ as stated in the proposition. Since $\cP$ is a slicing, the full subcategory $(\cP(>0), \cP(\leq 1))$ defines a bounded t-structure on $\mathcal{D}$. Let $\A=\cP((0, 1])$ be the heart of this t-structure, then $\A$ is an abelian category.  Then the definition of the triple $\sigma=(Z, \cP, \phi(E)|_{E\in \ker(Z)\cap\cP((0, 1])})$ implies the weak see-saw property and defines the phases $\{\phi(E):E\in \ker(Z)\cap \A\}$.
\end{proof}

	We also modify the criterion Proposition 2.4 in \cite{StabTC} for the existence of HN filtration, for use in later sections.
	\begin{prop}
		\label{Prop:HNfil}
  Let $\sigma=(Z, \A, \phi(E)|_{E\in \ker(Z)\cap\A})$ be a triple satisfying conditions (i) and (ii) in the Definition \ref{Def:Weak}.
		Suppose $Z$ satisfies the following chain conditions:
		
		(i) There are no infinite sequences of subobjects in $\A$:
		$$\cdots\subset E_{i+1}\subset E_i\subset\cdots\subset E_1$$
		with $\phi(E_{i+1})>\phi(E_i/E_{i+1})$ for all $i$.
		
		(ii) There are no infinite sequences of quotients in $\A$:
		$$E_1\twoheadrightarrow \cdots\twoheadrightarrow E_i\twoheadrightarrow E_{i+1}\twoheadrightarrow\cdots$$
		with $\phi(K_i)>\phi(E_{i+1})$ for all $i$, where $K_i=\ker(E_i\rightarrow E_{i+1})$.

		Then $\sigma$ has the Harder-Narasimhan (HN) property.
	\end{prop}
	\begin{proof}
		We follow Bridgeland's argument with certain modifications. 
   
		Claim 1:  For any $0\neq E\in \A$ where $E$ is not semistable, 
		there exists a semistable subobject $A\subset E$ with $\phi(A)> \phi(E/A)$; similarly, there exists semistable quotient $E\twoheadrightarrow B$ with $\phi(K)> \phi(B)$ where $K=\kernel (E\twoheadrightarrow B)$.  
		
		We prove the claim for semistable quotient, as the case of semistable subobject is similar. 
  We have a sequence of short exact sequences 
		$$0\rightarrow A_i\rightarrow E_{i-1}\rightarrow E_{i}\rightarrow 0,$$ with $E_0=E$, and $\phi(A_i)>\phi(E_i)$ for all $i$. 
		By the chain condition (ii), we have the sequence  $$E\twoheadrightarrow E_1\twoheadrightarrow ...\twoheadrightarrow E_i\twoheadrightarrow...$$
		is stationary, hence there exists semistable quotient $E\twoheadrightarrow E_n$. Let $A^{(i)}$ be the kernel of $E\twoheadrightarrow E_i$. Next we show that $\phi(A^{(n)})>\phi(E_n)$.

  By the diagram
		\begin{center}
		\begin{tikzcd}
			& & & 0\ar{d} &\\
			& 0	\ar{d}&  &A_i \ar{d} &  \\
			0	\ar{r} & A^{(i-1)} \ar{d}\ar{r} &E \ar[equal]{d}\ar{r} & E_{i-1}\ar{d}\ar{r} &0\\
			0\ar{r}&A^{(i)} \ar{r}& E \ar{r} & E_i\ar{r}\ar{d}& 0\\
			&& &0 &
		\end{tikzcd}
		\end{center}
		we have short exact sequence 
		$$0\rightarrow A^{(i-1)}\rightarrow A^{(i)}\rightarrow A_i\rightarrow 0.$$
		We prove by induction that $\phi(A^{(i)})>\phi(E_{i})$. By induction hypothesis we have $\phi(A^{(i-1)})>\phi(E_{i-1})$, which implies that $$\phi(A^{(i-1)})>\phi(E_{i-1})\geq \phi(E_i).$$ We also know that $\phi(A_i)>\phi(E_i)$. By Lemma \ref{lem:AG52-37-1}, we have $\phi(A^{(i)})>\phi(E_i)$.  This finishes the proof of Claim 1. 
		
		A maximally destabilising quotient (mdq) of an object $0\neq E\in\A$ is
		defined to be a nonzero quotient $E\twoheadrightarrow B$ such that any nonzero semistable quotient
		$E\twoheadrightarrow B'$ satisfies $\phi(B')\geq \phi(B)$, with equality holding only if (i)
		$E\twoheadrightarrow B'$ factors via $E\twoheadrightarrow B$ and (ii) $\phi(K')\leq\phi(B')$ for $K'=\kernel (B\rightarrow B')$. A routine argument shows  $B$ is semistable (the proof requires the use of condition (ii), which we did not need to impose in the case of Bridgeland stability conditions), and $\phi(E)\geq \phi(B)$.
		
		We first check the mdq exists. If $E$ is semistable, then itself is the mdq for $E$. If $E$ is not semistable, then by Claim 1, there is a short exact sequence in $\A$: 
		
		$$0\rightarrow A\rightarrow E\rightarrow E'\rightarrow 0. $$
		with $A$ semistable and $\phi(A)>\phi(E')$. 
  
  Claim 2:  If $E'\twoheadrightarrow B$ is a mdq for $E'$, then the composition $E\twoheadrightarrow B$ is a mdq for $E$. 
  
		Indeed, let $E\twoheadrightarrow B'$ be a quotient with $B'$ semistable. If $\phi(B')\leq \phi(B)$, then $\phi(B')\leq\phi(E')<\phi(A)$. Then there is no map from $A$ to $B'$, and hence the map $E\rightarrow B'$ factors through $E'$. This implies that $\phi(B')\geq \phi(B)$, and hence $\phi(B')= \phi(B)$. Since $E'\twoheadrightarrow B$ is a mdq for $E'$, the conclusion of the Claim follows. 
  
  We continue the process by replacing $E$ by $E'$. By the chain condition (ii), this process terminates. So every non-zero object have a mdq. 
		
		Now we prove the existence of HN filtration under the assumption of the Proposition. Take $0\neq E\in\A$, if $E$ is semistable, we are done. If not, there is a short exact sequence in $\A$:
		$$0\rightarrow E'\rightarrow E\rightarrow B\rightarrow0,$$  
		with $E\twoheadrightarrow B$ the mdq for $E$. By definition of mdq, $B$ is semistable. Also by the construction in the previous paragraph, $\phi(E')>\phi(B)$. Let $E'\rightarrow B'$ be the mdq of $E'$. Consider the diagram $(\dagger)$ in the proof of  \cite[Proposition 2.4]{StabTC}, in which $K$ is taken to be the kernel of $E' \to B'$, and $Q$ taken to be the cokernel of $K \to E$:
		
		\begin{equation}
			\label{Equ:Dagger}
			\begin{tikzcd}
				& & 0\ar{d} & 0\ar{d} &\\
				0\ar{r}& K	\ar{r}\ar[equal]{d}& E'\ar{r}\ar{d} &B' \ar{r}\ar{d} & 0 \\
				0\ar{r} & K\ar{r}&E \ar{r}{a} \ar{d}{c}&Q\ar{r}\ar{d}{b} &0\\
				& & B \ar[equal]{r}\ar{d} & B\ar{r}\ar{d}& 0\\
				& & 0&0 &
			\end{tikzcd}
		\end{equation}
		
		Claim 3. We have $\phi(B')>\phi(B)$. 
  
  First consider the case $Q$ is semistable. Since $E\twoheadrightarrow B$ is mdq, we have $\phi(Q)\geq\phi(B)$. But since $Q$ is semistable, we have $\phi(Q)\leq\phi(B)$. So $\phi(Q)=\phi(B)$. By the property of mdq, the map $E\rightarrow Q$ factors through $E\twoheadrightarrow B$, contradicting with $E'\rightarrow Q$ is nonzero. 

From here on, assume  $Q$ is not semistable. Let $d : Q\twoheadrightarrow Q'$ be the mdq of $Q$.  Since $Q'$ is the mdq for $Q$, we have $\phi (B) \geq \phi (Q')$; since $B$ is the mdq for $E$, we have $\phi (Q') \geq \phi (B)$.  Hence $\phi (B)=\phi (Q')$.    Now, that $Q'$ is the mdq for $Q$ implies $b$ factors as $b=ed$ for some $e : Q' \twoheadrightarrow B$; on the other hand, that $B$ is the mdq for $E$ implies that $da$ factors as $fc$ for some $f : B \twoheadrightarrow Q'$.  Overall, we have $ed=b$ and $da=fc$ which gives
\begin{gather*}
  (fe)da = fba = fc = da.
\end{gather*}
Since $da$ is surjective, this means  $e$ is injective, and hence an isomorphism, i.e.\ $Q' \cong B$. 

If $\phi (Q) > \phi (Q')=\phi (B)$. Applying the weak seesaw property to the short exact sequence $0 \to B' \to Q \to B \to 0$, we obtain $\phi (B') \geq \phi(Q) > \phi (B)$ and hence $\phi (B')>\phi (B)$.

Assume $\phi(Q)=\phi(B)$. Since $Q$ is not semistable, there is a short exact sequence in $\A$
\[
0\to M\to Q\to M'\to 0
\] 
such that $M'$ is semistable and $\phi(M)>\phi(M')$. Since $\phi(Q)\geq \phi(M')$, $\phi(M')\geq \phi(B)$, we have $\phi(M')=\phi(B)$. Then the map $Q\to M'$ factors through $Q\to B$. Consider the following diagram:
	\begin{equation}
			\begin{tikzcd}
				& & 0\ar{d} & 0\ar{d} &\\
				0\ar{r}& B'	\ar{r}\ar[equal]{d}& M\ar{r}\ar{d} &N \ar{r}\ar{d} & 0 \\
				0\ar{r} & B'\ar{r}&Q \ar{r} \ar{d}&B\ar{r}\ar{d} &0\\
				& & M' \ar[equal]{r}\ar{d} & M'\ar{r}\ar{d}& 0\\
				& & 0&0 &
			\end{tikzcd}
		\end{equation} 
By the definition of mdq, we have $\phi(N)\leq \phi(M')$. Then $\phi(M)>\phi(M')$ implies that $\phi(B')>\phi(M')=\phi(B)$.

		Repeating the process for $E'$, by chain condition (i) the sequence terminates, and this is the HN filtration for $E$. 
	\end{proof}
	\begin{defn}
	 Let $E$ be a $\sigma$-semistable object for some weak stability condition 
   \[\sigma=(Z_\sigma, \cP, \{\phi_\sigma(K)\}|_{K\in \ker(Z_\sigma)\cap\cP(0, 1]}).
 \]
 A Jordan-H\"{o}lder filtration of $E$ is a filtration 
	\begin{equation*}
		0=E_0\subset E_1\subset...\subset E_n=E
	\end{equation*}
	such that the factors $gr_i(E)=E_i/E_{i-1}$ are stable, and 
	for each $i$ either $gr_i(E)\in \ker(Z_\sigma)$, or 
	$\phi_\sigma(gr_i(E))=\phi_\sigma(E)$.
\end{defn}

 \begin{rem}
 \label{Rem:stableness}
 Recall that in \ref{para:ssdefforweakstabcon}, we defined when an object is $\sigma$-stable in the abelian category $\A=\cP((0, 1])$. For a general $t$ we say an object is stable in $\cP(t)$ if it is a shift of a stable object in $\cP(t')$ for some $0<t'\leq 1$.	
 \end{rem}
 
In Section \ref{sec:fourtypesofweakstab}, we will give four types of examples of weak stability conditions:
\begin{itemize}
    \item[(i)]  when the ample class in $\sigma_{\omega, B}$ degenerates to nef class;
    \item[(ii)]  when the coefficient of $\ch_0$ in the central charge $Z$ of a Bridgeland stability condition degenerates to zero;
    \item[(iii)] when (i) and (ii) occur at the same time; and
    \item[(iv)] when a relative Fourier-Mukai transform is applied to  the weak stability condition in (iii).
\end{itemize}   
 These are natural classes of weak stability conditions we obtain when specific parameters of Bridgeland stability approach certain limits. One direct application of the these four weak stability conditions is in the sequel to this  article, in which we obtain stability of line bundles at Bridgeland stability conditions by first studying their stability  at these four types of weak stability conditions.


\section{Bogomolov-Gieseker type inequalities}\label{sec:BGi}

Since we will be dealing with weak stability conditions that arise from nef divisors in  Section \ref{sec:fourtypesofweakstab}, we need to study Bogomolov-Gieseker-type inequalities for slope semistable sheaves with respect to nef divisors first.

Let $X$ be a smooth projective surface, and $\omega$ an ample $\mathbb{R}$-divisor on $X$.  The usual Bogomolov-Gieseker inequality for sheaves says the following: for any $\mu_\omega$-semistable torsion-free sheaf $E$ on $X$, we have
\[
  \ch_2 (E) \leq \frac{\ch_1(E)^2}{2\ch_0(E)}.
\]
It is easily checked that this inequality is preserved under twisting the Chern character by a $B$-field.

When $X$ is a K3 surface, we have the following, stronger Bogomolov-Gieseker inequality which  is known to experts (e.g.\ see \cite[Section 6]{ABL}):


\begin{prop}\label{prop:K3surfstrongerBG}
Let $X$ be a K3 surface, and $\omega, B$ any $\mathbb{R}$-divisors on $X$ where $\omega$ is ample.  Then for any $\mu_\omega$-stable torsion-free sheaf $E$ on $X$, we have
\begin{equation}\label{eq:K3surfstrongerBG}
  \ch_2^B(E) \leq \frac{(\ch_1^B(E))^2}{2\ch_0(E)} - \ch_0(E) + \frac{1}{\ch_0(E)}.
\end{equation}
\end{prop}


We can extend the above result to a certain class of nef divisors on a surface, by first proving an `openness of stability' result.  Recall that on a smooth projective surface, a divisor class is nef if and only if it is movable.



\begin{prop}\label{prop:K3surfstrongerBGdeform}
Let $X$ be a smooth projective surface.  Suppose $H, H'$ are nef $\mathbb{R}$-divisors on $X$ such that
\begin{itemize}
\item[(i)] The value
\[
  \min{ \{Hc_1(E) : Hc_1(E)>0, E \in \mathrm{Coh}(X)\}}
\]
exists.
\item[(ii)] There exists $\epsilon_1>0$ such that $H+\epsilon H'$ is nef for all $\epsilon \in (0, \epsilon_1)$.
\end{itemize}
Then for any torsion-free sheaf $E$ on $X$, there exists $\epsilon_0 \in (0, \epsilon_1)$ depending only on $E, H$ and $H'$  such that, for any $0< \epsilon < \epsilon_0$, we have
\[
   \text{$E$ is $\mu_H$-stable} \Rightarrow \text{ $E$ is  $\mu_{H+\epsilon H'}$-stable}.
\]
\end{prop}

\begin{proof}
Since $H'$ is nef, and hence movable, for any nonzero proper subsheaf $A$ of $E$, we have $\mu_{H'}(A) \leq \mu_{H', \mathrm{max}}(E)$ \cite[Corollary 2.24]{greb2016movable}, where $\mu_{H', \mathrm{max}}(E)$ depends only on $H'$ and $E$.

On the other hand, by  assumption (i) and the $\mu_H$-stability of $E$, there exists $\delta >0$ depending only on $E$ and $H$ such that
\[
  \max{\biggl\{ \mu_H(A)= \frac{H\ch_1(A)}{\ch_0(A)} : A \text{ is a nonzero proper subsheaf of }E \biggr\}} \leq \mu_H(E) -\delta.
\]

Now fix an $\epsilon_0 \in (0, \epsilon_1)$ such that
\[
  \epsilon_0 \left|\mu_{H',\mathrm{max}}(E)\right| < \frac{\delta}{2}  \text{\quad and \quad} \epsilon_0 |\mu_{H'}(E)| < \frac{\delta}{2}.
\]
Note that $\epsilon_0$ depends only on $E, H'$ and $\delta$.  Then for any  $0 < \epsilon < \epsilon_0$ and any nonzero proper subsheaf $A$ of $E$, we have
\begin{align*}
  \mu_{H+\epsilon H'}(A) &= \mu_H(A) + \epsilon \mu_{H'}(A) \\
  &\leq (\mu_H(E) - \delta) + \epsilon \mu_{H',\mathrm{max}}(E) \\
  &< \mu_H(E) -\delta +\frac{\delta}{2} \\
  &=\mu_H(E) - \frac{\delta}{2} \\
  &< \mu_H(E) + \epsilon \mu_{H'}(E) \\
  &= \mu_{H+\epsilon H'}(E).
\end{align*}
This means that, for any $0 < \epsilon < \epsilon_0$, the sheaf $E$ is also  $\mu_{H + \epsilon H'}$-stable.

Note that the purpose of condition (ii) is to ensure that when we speak of $\mu_{H+\epsilon H'}$-stability above, we are using a divisor $H+\epsilon H'$ that is  nef and not merely an arbitrary curve class.
\end{proof}

Proposition \ref{prop:K3surfstrongerBGdeform} is similar to the result \cite[Theorem 3.4]{greb2016movable} by Greb, Kebekus, and Peternell, but with slightly different assumptions.

\begin{cor}\label{cor:K3surfstrongerBGdeform}
Let $X$ be a K3 surface.  Suppose $H, H'$ are nef $\mathbb{R}$-divisors on $X$ satisfying condition (i)  in Proposition \ref{prop:K3surfstrongerBGdeform} as well as 
\begin{itemize}
\item[(ii')] There exists $\epsilon_1>0$ such that $H+\epsilon H'$ is ample for all $\epsilon \in (0, \epsilon_1)$.
\end{itemize}
  Then any $\mu_H$-stable torsion-free sheaf $E$ on $X$ satisfies the inequality \eqref{eq:K3surfstrongerBG}.
\end{cor}

\begin{proof}
Let $\epsilon_0$ be as in Proposition \ref{prop:K3surfstrongerBGdeform}.  Then for any $0 < \epsilon < \epsilon_0$, the sheaf $E$ is also $\mu_{H+\epsilon H'}$-stable.  Since $H+\epsilon H'$ is an ample divisor, Proposition \ref{prop:K3surfstrongerBG} applies to $E$ and the claim follows.
\end{proof}


Here is a lemma we need for the example below:

\begin{lem}\label{lem:AG50-132-1}
Let $p : X \to Y$ be an elliptic surface with a section $\Theta$.  Then a divisor of the form $\Theta + af$ where $a \in \mathbb{R}_{>0}$ is ample if and only if $a > -\Theta^2$.
\end{lem}

\begin{proof}
By the Nakai-Moishsezon criterion, $\Theta + af$ is ample if and only if $(\Theta + af)^2=2a-e>0$ and $(\Theta + af)C>0$ for every irreducible curve $C \subset X$.

Suppose $C$ is an irreducible curve on $X$ that is distinct from $\Theta$.  Then $\Theta C\geq 0$ and $fC \geq 0$.  If $fC=0$, then $C$ is a vertical divisor, in which case $\Theta C>0$ and $(\Theta +af)C>0$.  If $fC>0$, then $(\Theta +af)C>0$ as well.

Now suppose  $C=\Theta$.  Then $(\Theta + af)C=a+\Theta^2$.  Since $a>0$, it now follows that $\Theta + af$ is ample if and only if $a>-\Theta^2$.
\end{proof}

\begin{eg}\label{eg:K3nefdivBGineq}
Let $X$ be a Weierstra{\ss} elliptic K3 surface (i.e.\ $e=2$).  If we take the nef divisors
 \[
  H = \Theta + ef, \text{\quad} H' = f
\]
in Proposition \ref{prop:K3surfstrongerBGdeform}, then condition (i) is clearly satisfied since $H$ is integral, while condition (ii') in Corollary \ref{cor:K3surfstrongerBGdeform} is satisfied by Lemma \ref{lem:AG50-132-1}.  Therefore, by Corollary \ref{cor:K3surfstrongerBGdeform}, the sharper Bogomolov-Gieseker inequality \eqref{eq:K3surfstrongerBG} holds for $\mu_{\Theta+ef}$-stable torsion-free sheaves on $X$ even though $\Theta+ef$ is a nef divisor.
\end{eg}

\section{Four types of weak stability conditions} \label{sec:fourtypesofweakstab}
In this section, we introduce four natural classes of weak stability conditions. 
Let $p : X \to Y=\mathbb{P}^1$ be a Weierstra{\ss} elliptic K3 surface throughout this section, so $e = -\Theta^2=2$.  Let $i : \mathbb{P}^1 \hookrightarrow X$ denote the canonical section. 

\paragraph[Weak stability condition at the origin]
	\label{section:weakstaborigin}   We consider the usual central charge (with zero $B$-field) as
\[
Z_\omega = -\ch_2 + V_\omega \ch_0 + i \omega \ch_1
\]
where $\omega = R_\omega (\Theta + (D_\omega + e)f)$.  Note that according to our definitions of $R_\omega, D_\omega, V_\omega$ for $\omega$ ample, we have
\[
  V_\omega = \tfrac{\omega^2}{2}, \text{\quad} R_\omega = \sqrt{\frac{V_\omega}{D_\omega + \tfrac{e}{2}}}.
\]

Rescaling the imaginary part of $Z_\omega$, we obtain the central charge 
\begin{equation}
\label{Equ:CentralChargeVD}
Z_{V_\omega, D_\omega}(E)=-\ch_2(E)+V_\omega\ch_0(E)+i(\Theta+(D_\omega+e)f)\cdot \ch_1(E).
\end{equation}
We consider the limit of $Z_{V_\omega, D_\omega}$ when $D_\omega, V_\omega \to 0$
 \[
   Z_{H} = -\ch_2 + i H \ch_1
 \]
 where $H = \Theta + ef$.  We will now construct a heart $\Bc_{H,k}^0$ on which $Z_H$ is a weak stability function.  (We defer the verification of the HN property to the end of this section.)

  Recall that for a proper morphism $f : X \to Y$ on smooth varieties, the Grothendieck-Riemann-Roch (GRR) theorem  says
 \[
   \ch( f_\ast \alpha).\mathrm{td}(T_Y) = f_\ast (\ch(\alpha).\mathrm{td}(T_X))
 \]
 for any $\alpha \in K(X)$.  Note that
 \[
  \mathrm{td}(T_{\mathbb{P}^1}) = 1 + \tfrac{1}{2}c_1(T_{\mathbb{P}^1}) = 1 + [pt]
 \]
from \cite[Example II8.20.1]{Harts}, while $\mathrm{td}(T_X) = (1,0, 2[pt])$ since $X$ is a K3 surface.  Therefore, GRR gives
\[
  [X] \cap \left( \ch (i_\ast \OO_{\mathbb{P}^1}(m)).(1,0,2[pt]) \right) = i_\ast \left( (\ch  (\OO_{\mathbb{P}^1}(m)).(1,[pt])) \cap [Y] \right)
\]
from which we obtain
\begin{align*}
  \ch_0(i_\ast \OO_{\mathbb{P}^1}(m)) &= 0\\
  \ch_1(i_\ast \OO_{\mathbb{P}^1}(m)) &= i_\ast [Y] = \Theta \\
  \ch_2(i_\ast \OO_{\mathbb{P}^1}(m)) &= (m+1)[pt]
\end{align*}

Proceeding as in Tramel-Xia \cite{tramel2017bridgeland}, we set
\begin{align*}
  \Tc^a_H &= \langle E \in \Coh (X) : E \text{ is $\mu_H$-semistable}, \mu_{H, \mathrm{min}}(E)>a \rangle \\
  \Fc^a_H &= \langle E \in \Coh (X) : E \text{ is $\mu_H$-semistable}, \mu_{H, \mathrm{max}}(E) \leq a \rangle.
\end{align*}
Even though $H$ is a nef divisor that is not ample, it is a movable class on $X$, and so the notions  $\mu_{H,\mathrm{min}}$ and $\mu_{H, \mathrm{max}}$ make sense \cite[Corollary 2.26]{greb2016movable}. Let 
\[
\A^a_H=\langle \F^a_H[1], \T^a_H\rangle.
\]
Next, for any integer $k$ we set
\[
  \Fc^a_{H,k} = \langle \OO_\Theta (i) : i \leq k \rangle
\]
where $\OO_\Theta (i)$ denotes $i_\ast \OO_{\mathbb{P}^1}(i)$, and set
\[
  \Tc^a_{H,k} = \{ E \in \Ac^a_{H} : \Hom (E, \OO_\Theta (i)) = 0 \text{ for all } i \leq k \}.
\]
From \cite[Lemma 3.2]{tramel2017bridgeland}, we know $(\Tc^a_{H,k}, \Fc^a_{H,k})$ is a torsion pair in $\Ac^a_{H}$, allowing us to perform a tilt to obtain the heart
\[
  \Bc^a_{H,k} = \langle \Fc^a_{H,k}[1], \Tc^a_{H,k} \rangle.
\]

\begin{lem}\label{lem:AG51-114}
Let $X$ be a Weierstra{\ss} elliptic surface with canonical section $\Theta$ and $e=-\Theta^2>0$.  Suppose $C \subseteq X$ is an irreducible curve not contained in $\Theta$.  Then $C.(\Theta + ef)>0$.
\end{lem}

\begin{proof}
If $C$ is a vertical divisor, then $C.f=0$ and $C.\Theta >0$, in which case the lemma follows.

Suppose $C$ is a horizontal divisor not contained in $\Theta$.  Then $C.\Theta \geq 0$, while $C.f>0$ by \cite[Lemma 3.15]{Lo11}, in which case the lemma also follows.
\end{proof}

We also know that the Bogomolov-Gieseker inequality holds for slope semistable sheaves with respect to a movable class on a surface: When $X$ is a smooth projective surface and $\alpha$ is a nonzero movable class on $X$ (e.g.\ when $\alpha$ is nef), for any $\mu_\alpha$-semistable torsion-free sheaf $E$ we have the usual Bogomolov-Gieseker inequality
\begin{equation}\label{eq:HIT-surf-mov}
  \ch_2(E) \leq \frac{\ch_1(E)^2}{2\ch_0(E)},
\end{equation}
where the inequality is strict if $E$ is not locally free \cite[Theorem 5.1]{greb2016movable}.

\begin{prop} \label{prop:weakstabfuncbefore}
Let $X$ be a Weierstra{\ss} elliptic K3 surface.  Then $Z_H$ is a weak stability function on the heart $\Bc^0_{H,k}$ for $k=-1, -2$.
\end{prop}
\begin{proof}
Take any $E \in \Bc^0_{H,k}$. It suffices to show that $Z_H (E) \in \mathbb{H}_0$ for either $E \in \Tc^0_{H,k}$ or $E \in \Fc^0_{H,k}$, when $k=-1, -2$.

Suppose $E \in \Tc^0_{H,k}$.  Then $E \in \Ac^0_H$, and it is easy to see that $Z_H(E) \in \mathbb{H}$ when we are in one of the following cases:
\begin{itemize}
\item $E \in \Tc^0_H$ and $\ch_0(E)>0$.
\item $E \in \Tc^0_H$ is supported in dimension 1, and its support is not contained in $\Theta$ (use Lemma \ref{lem:AG51-114}).
\item $E\in \Tc^0_H$ is supported in dimension 0.
\item $E = F[1]$ where $F$ is a $\mu_H$-semistable torsion-free sheaf with $\mu_H(F)<0$.
\end{itemize}
When $E = F[1]$ where $F$ is a $\mu_H$-semistable torsion-free sheaf with $\mu_H(F)=0$, we have
\begin{align*}
  Z_H (E) &= Z_H(F[1]) \\
  &= \ch_2 (F) \\
  &\leq \frac{\ch_1(F)^2}{2\ch_0(F)} \text{ by the Bogomolov-Gieseker inequality for $\mu_H$-semistability \cite[Theorem 5.1]{greb2016movable}} \\
  &\leq 0 \text{ by the Hodge Index Theorem \cite[Theorem 7.14]{McKnotes}}.
\end{align*}
Note that the Hodge Index Theorem applies in the last step because $H^2 >0$.  Also, if $E$ is a 1-dimensional sheaf supported on $\Theta$, then because $\Hom (E, \OO_\Theta (i))=0$ for all $i\leq k$, it follows that $E \in \langle \OO_\Theta (i) : i > k \rangle$, which gives $Z_H(E) \in \mathbb{R}_{\leq 0}$.  Since every object in $\Tc^0_{H,k}$ is an extension of objects of the forms above, we have $Z_H(\Tc^0_{H,K})\subseteq \mathbb{H}_0$.

Lastly, if $E \in \Fc^0_{H,k}$, then $E \in \langle \OO_\Theta (i) : i \leq k \rangle [1]$, in which case we also have $Z_H(E) \in \mathbb{H}_0$.
\end{proof}

\begin{prop}\label{prop:B0HkcapkerZH}
Let $X$ be a Weierstra{\ss} elliptic K3 surface.  Then
\[
  \Bc^0_{H,k} \cap \kernel Z_H = \begin{cases}
  \langle  \OO_\Theta (-1) [1], \OO_X [1] \rangle &\text{ if $k=-1$} \\
  \langle \OO_X [1], \OO_\Theta (-1)\rangle &\text{ if $k=-2$}
  \end{cases}.
\]
\end{prop}

\begin{proof}
For convenience, we will write $\Ac$ and $\Bc$ to denote the hearts $\Ac^0_H$ and $\Bc^0_{H,k}$, respectively, in this proof.

Take any nonzero object $E \in \Bc^0_{H,k} \cap \kernel Z_H$.  Then $Z_H (\Hc^{-1}_\Ac (E))=0=Z_H(\Hc^0_\Ac (E))$.   In particular, we have $\ch_2 (\Hc^{-1}_\Ac (E))=0$.  Since $\Hc^{-1}_\Ac (E) \in \langle \OO_\Theta (i) : i \leq k \rangle$, this means that
\[
  \Hc^{-1}_\Ac (E) \in \begin{cases}
  \langle \OO_\Theta (-1) \rangle &\text{ if $k=-1$} \\
  \{0\} &\text{ if $k=-2$}
  \end{cases}.
\]
Also, since $\Im Z_H (\Hc^0_\Ac (E))=0$ and $\Im Z_H$ is nonnegative on $\Ac$, it follows that $\Im Z_H (H^0(\Hc^0_\Ac (E)))$ and $\Im Z_H (H^{-1}(\Hc^0_\Ac (E)))$ are both zero.  As a result,  the sheaf $H^0(\Hc^0_\Ac (E))$ must be a torsion sheaf  supported on $\Theta$ by Lemma \ref{lem:AG51-114}, while $M := H^{-1}(\Hc^0_\Ac (E))$ is a $\mu_H$-semistable torsion-free sheaf with $\mu_H=0$.

Since $H$ is a nef divisor,  the usual Bogomolov-Gieseker inequality holds for  $\mu_H$-semistable sheaves \cite[Theorem 5.1]{greb2016movable}
\[
  \ch_2 (M) \leq \frac{c_1(M)^2}{2\ch_0(M)}.
\]
Also, since $H$ is nef with $H^2=e>0$, the Hodge Index Theorem gives $c_1(M)^2 \leq 0$, and so we have $\ch_2(M) \leq 0$.  On the other hand, by the construction of $\Bc$, we have
\[
  \Hom (H^0(\Hc^0_\Ac (E)), \OO_\Theta (i)) =0 \text{ for any $i \leq k$}
\]
which means that $H^0(\Hc^0_\Ac (E))$ is an extension of a sheaf in $\langle \OO_\Theta (i) : i > k \rangle$ by a 0-dimensional sheaf, giving us $\ch_2 (H^0 (\Hc^0_\Ac (E)))> k+1 \geq 0$ (recall $k$ is either $-1$ or $-2$).    Since $\Re Z_H (\Hc^0_\Ac (E))=0$, this forces
\[
\ch_2 (M)=\ch_2 (H^0(\Hc^0_\Ac (E)))=0.
\]
In particular, $H^0(\Hc^0(\Ac (E)))$ is forced to be a pure 1-dimensional sheaf, and
\[
  H^0(\Hc^0_\Ac (E)) \in \begin{cases} \{0\} &\text{ if $k=-1$} \\
  \langle \OO_\Theta (-1)\rangle &\text{ if $k=-2$}
  \end{cases}.
\]


We now show that, in fact, $M$ lies in $\langle \OO_X \rangle$.  Let $M_i$ be the $\mu_H$-Jordan H\"{o}lder factors of $M$ (which exist by \cite[Corollary 2.27]{greb2016movable}).  By Example \ref{eg:K3nefdivBGineq}, each $M_i$ satisfies the sharper Bogomolov-Giesker inequality \eqref{prop:K3surfstrongerBG}.  Together with the Hodge Index Theorem, for each $i$ we have
\[
\ch_2(M_i) \leq -\ch_0(M_i) + \frac{1}{\ch_0(M_i)}
\]
where the right-hand side is strictly less than 0 if $\ch_0(M_i)>1$, while equal to 0 if $\ch_0(M_i)=0$.  Since $0=\ch_2(M)=\sum_i \ch_2(M_i)$, it follows that each $M_i$ is a rank-one torsion-free sheaf  with $\ch_2(M_i)=0$.  In fact, each $M_i$ must be a line bundle, for if $M_j$ is not locally free for some $j$, then
\begin{equation}\label{eq:AG51-174-1}
  \ch_2 (M_j) = \ch_2(M_j^{\ast\ast})-\ch_2(M_j^{\ast\ast}/M_j) < \ch_2(M_j^{\ast\ast})\leq  \frac{c_1(M_j^{\ast\ast})^2}{2\ch_0(M_j^{\ast\ast})} \leq 0
\end{equation}
where the second last inequality follows from the usual
 Bogomolov-Gieseker inequality, and the last inequality follows from the Hodge Index Theorem for $H$.  This contradicts  $\ch_2(M_j)=0$, and so each $M_i$ is a line bundle.  Moreover, equality holds in the Hodge Index Theorem in \eqref{eq:AG51-174-1}, and so $c_1(M_i)\equiv 0$; since $X$ is a K3 surface, this  implies $c_1(M_i)=0$ and hence $M_i \cong \OO_X$.  That is, $M \in\langle \OO_X \rangle$.

So far, we have shown that any object $E$ in $\Bc^0_{H,k} \cap \kernel Z_H$ fits in an exact triangle in $D^b(X)$
\[
E' \to E \to E'' \to E'[1]
\]
where
\begin{itemize}
\item $E' \in \langle \OO_\Theta (-1)[1]\rangle$ and $E'' \in \langle \OO_X[1]\rangle$  if $k=-1$;
\item $E' \in \langle \OO_X[1]\rangle$ and $E'' \in \langle \OO_\Theta (-1)\rangle$  if $k=-2$.
 \end{itemize}
 This proves the proposition. 
\end{proof}

We actually have that the objects in $\B_{H, -1}^{0}\cap \ker(Z_H)$ are direct sums of $\cO_\Theta(-1)[1]$ and $\cO_X[1]$, and objects in $\B_{H, -2}^{0}\cap \ker(Z_H)$ are direct sums of $\cO_\Theta(-1)$ and $\cO_X[1]$, as shown in the following lemma.

\begin{lem}\label{lem:AG51-174-1}
Let $X$ be a Weierstra{\ss} elliptic K3 surface.  Then
\[
  \Ext^1 (\OO_\Theta (-1), \OO_X[1]) = 0 = \Ext^1 (\OO_X[1], \OO_\Theta (-1)[1]),
\]
\[
\Ext^1(\cO_\Theta(-1), \cO_\Theta(-1))=0,
\]
and every object in $\Bc^0_{H,k} \cap \kernel Z_H$ is a direct sum of $\OO_\Theta (-1) [1]$ and $\OO_X [1]$ (resp.\ $\OO_X [1]$ and $\OO_\Theta (-1)$) if $k=-1$ (resp.\ $k=-2$).
\end{lem}

\begin{proof}
By Serre Duality,
\[
\Ext^1 (\OO_\Theta (-1), \OO_X [1]) \cong \Hom (\OO_X, \OO_\Theta (-1))
\]
which must vanish, because every morphism $\OO_X \to \OO_\Theta (-1)$ of sheaves factors as $\OO_X \to \OO_\Theta \to \OO_\Theta (-1)$, but the only morphism $\OO_\Theta \to \OO_\Theta (-1)$ is the zero map (e.g.\ use slope stability for torsion-free sheaves on $\Theta$).

With the above vanishing and $\Ext^2 (\OO_X, \OO_\Theta (-1)) \cong \Hom (\OO_\Theta (-1), \OO_X)=0$, it follows that
\begin{align*}
 \dimension \Ext^1 (\OO_X, \OO_\Theta (-1)) &= -\chi (\OO_X, \OO_\Theta (-1)) \\
&= (v(\OO_X), v(\OO_\Theta (-1))) \text{ by Riemann-Roch} \\
&= ((1, 0, 1), (0, \Theta, 0)) \\
&= 0.
\end{align*}
		Also by Serre duality, we have $$\Ext^2(\cO_\Theta(i), \cO_\Theta(i))\simeq\Hom(\cO_\Theta(i), \cO_\Theta(i))\simeq \mathbb{C}.$$

By Riemann-Roch, we have $$\chi(\cO_\Theta(i), \cO_\Theta(i))=-(\nu(\cO_\Theta(i)), \nu(\cO_\Theta(i)))=2.$$
Hence $2-\mathrm{ext}^1(\cO_\Theta(i), \cO_\Theta(i))=2$, which implies that $$\Ext^1(\cO_\Theta(i), \cO_\Theta(i))=0.$$

The second part of the proposition follows easily from Proposition \ref{prop:B0HkcapkerZH}.
\end{proof}

	
\paragraph \label{para:weakstab1-ker-phases} 
To define a weak stability function, we still need to define the phases of objects in the subcategory $\B_{H, k}^0\cap \ker(Z_H)$ in such a way that the weak see-saw property is satisfied. We define the phases $\{\phi_{H, k}(K)|_{K\in\B_{H, k}^0\cap \ker(Z_H)}\}$ by taking the limit of the phase of $Z_{V_\omega, D_\omega}(K)$ as $V_\omega, D_\omega\to 0$.


	We have $$\lim\limits_{V_\omega, D_\omega\rightarrow 0}\phi_{V_\omega, D_\omega}(\cO_X[1])=1$$ and $$\lim\limits_{V_\omega, D_\omega\rightarrow 0}\phi_{V_\omega, D_\omega}(\cO_\Theta(-1))=\frac{1}{2}.$$ 
 As a result, it is only when $k=-2$ that we have $0<\phi_{H, k}(K)\leq 1$ for any $K\in\Bc_{H,k}^0 \cap \ker (Z_H)$.
 Given a point $b\in \mathbb{P}^1$, whose coordinates are given by $[V_\omega: D_\omega]$, we define 
\begin{equation*}
    \phi^L_b(K)=\lim\limits_{V_\omega, D_\omega\to 0, [V_\omega: D_\omega]=b}\phi_{V_\omega, D_\omega}(K).
\end{equation*}

We check the weak see-saw property in Definition \ref{Def:Weak}.


\begin{prop}
\label{Prop:weakseesawbefore}
Let 
\[
0\to K_1\to K\to K_2\to 0
\]
be a short exact sequence in $\B^0_{H, -2}\cap \ker(Z_H)$. Then we have 

(i) $\phi^L_b(K_1)\leq \phi^L_b(K_2)$ implies that $\phi^L_b(K_1)\leq \phi^L_b(K)\leq \phi^L_b(K_2)$.

(ii) $\phi^L_b(K_1)\geq \phi^L_b(K_2)$ implies that $\phi^L_b(K_1)\geq \phi^L_b(K)\geq \phi^L_b(K_2)$.
\end{prop}
\begin{proof}
By lemma \ref{lem:AG51-174-1}, we have for any $K\in \Bc_{H,-2} \cap \ker (Z_H)$, $K$ is isomorphic to the direct sum of $\cO_X[1]$ and $\cO_\Theta(-1)$. Assume that 
\[
K_0\simeq (\cO_\Theta(-1))^{l_0}\oplus(\cO_X[1])^{l_1},
\]
\[
K\simeq (\cO_\Theta(-1))^{m_0}\oplus(\cO_X[1])^{m_1},
\]
\[
K_1\simeq (\cO_\Theta(-1))^{n_0}\oplus(\cO_X[1])^{n_1}.
\]
To prove the proposition, it is enough to show the same relation for the Bridgeland slope. We define $\rho_{V_\omega, D_\omega}=-\frac{\Re(Z_{V_\omega, D_\omega})}{\Im(Z_{V_\omega, D_\omega})}$, and $\rho_b^L=-\text{cot}(\pi\phi_b^L(\_))$. We have
\begin{equation*}
\begin{split}
\rho^L_b(K)&=\lim\limits_{V_\omega, D_\omega\to 0, [V_\omega: D_\omega]=b}\rho_{V_\omega, D_\omega}(K)\\
&=\lim\limits_{V_\omega, D_\omega\to 0, [V_\omega: D_\omega]=b}-\frac{\Re(Z_{V_\omega, D_\omega}(K))}{\Im(Z_{V_\omega, D_\omega}(K))}\\
&=\lim\limits_{V_\omega, D_\omega\to 0, [V_\omega: D_\omega]=b}\frac{V_\omega\cdot m_1}{D_\omega\cdot m_0}
\end{split}
\end{equation*}
We write $b=V_\omega/D_\omega$, and $b\in[0, \infty]$. Then $\phi^L_b(K)=\frac{bm_1}{m_0}$. Similarly, we have $\phi^L_b(K_0)=\frac{bl_1}{l_0}$ and $\phi^L_b(K_1)=\frac{bn_1}{n_0}$.
Since $m_0=l_0+n_0$ and $m_1=l_1+n_1$, we obtain the weak see-saw property in Definition \ref{Def:Weak}.  
\end{proof}
For each $b\in\mathbb{P}^1$, we have a triple $\sigma^L_b=(Z_{H}, \B^0_{H, -2}, \{\phi_b(E)\}|_{E\in \ker (Z_{H})\cap\B^0_{H, -2}})$. Finally we check the HN property in Definition \ref{Def:Weak}.


\begin{prop}
\label{Prop:HNorigin}
	The triple $\sigma_b^L$ 
 satisfies the HN property, hence defines a weak stability condition.
\end{prop}
\begin{proof}
	We follow the argument of Proposition 7.1 in \cite{SCK3}.
		
 We first check condition (i) in  Proposition \ref{Prop:HNfil}. Assume we have such a chain of subobjects in $\B^0_{H, -2}$ as in condition (i). Consider the short exact sequence in $\B^0_{H, -2}$:
		
  \begin{equation}\label{eq:EEF-1}
  0\rightarrow E_{i+1}\rightarrow E_i\rightarrow F_i\rightarrow 0.
  \end{equation}
		Since $\Im(Z_H(F_i))\geq 0$, we have $\Im(Z_H(E_{i}))\geq \Im(Z_H(E_{i+1}))$. 
		Since $\Im(Z_H)$ has coefficients in $\mathbb{Z}$, we know that $\Im(Z_H)$ is discrete. Hence $\Im(Z_H(E_i))$ is constant for $i\gg 0$. Then $\Im(Z_H(F_i))=0$ for $i>n$ for some $n$. 
		
		Let $i>n$.  If $F_i\notin \B^0_{H, -2}\cap\ker(Z_H) $, then $\phi^L_b(F_i)=1$, contradicting  $\phi^L_b(E_{i+1})>\phi^L_b(F_i)$.
		
		So we must have $F_i\in \ker(Z_H)\cap \B^0_{H, -2}=\langle \cO_X[1],\cO_\Theta(-1)\rangle$ for $i>n$. 
By Lemma \ref{lem:AG51-174-1}, we have 
\[
\Ext^1(F_i, \cO_\Theta(-1))=\Ext^1( F_i,\cO_X[1])=0.
\]

Now, applying the functors $\Hom (-, \OO_X[1])$ and $\Hom (-,\OO_\Theta (-1))$ to \eqref{eq:EEF-1}, we see that  the dimensions of $\Hom(E_{i}, \cO_\Theta(-1))$ and  $\Hom(E_i, \cO_X[1])$ are both non-increasing as $i$ increases, and for $i\gg 0$, we have 
\[
\Hom (F_i, \OO_X[1])=\Hom (F_i, \OO_\Theta (-1)) = 0
\]
which implies $F_i=0$ by Proposition \ref{lem:AG51-174-1}. 


Next we consider condition (ii) in  Proposition \ref{Prop:HNfil}. Assuming we have a sequence 
\[
E=E_1\twoheadrightarrow E_2\twoheadrightarrow E_3\twoheadrightarrow...
\]

Consider the short exact sequence 
\begin{equation}\label{Equ:HNKEE}
0\to K_i\to E_i\to E_{i+1}\to 0.
\end{equation}
Similar to the argument above, we have $\Im(Z_H(E_i))$ is constant for $i\gg 0$.
Omitting a finite number of terms, we can assume there are short exact sequences 
\begin{equation}
\label{Equ:LEE}
0\to L_i\to E\to E_i\to 0
\end{equation}
with $\Im(Z_H(L_i))=0$ for all $i \geq 0$.
Also by the long exact sequence of cohomology, we have the surjective chain in $\Coh(X)$:
\[H^0(E_1)\twoheadrightarrow H^0(E_2)\twoheadrightarrow H^0(E_3)\twoheadrightarrow...\]
Since $\Coh(X)$ is noetherian, this sequence stabilizes. 
Hence we can assume that $H^0(E)\simeq H^0(E_i)$ for all $i$.
Taking cohomology sheaves of \eqref{Equ:LEE}, we have 
\begin{equation}
\label{Equ:4termexactseq}
0\to H^{-1}(L_i)\to H^{-1}(E)\xrightarrow{f} H^{-1}(E_i)\to H^0(L_i)\to 0.
\end{equation}
Consider the short exact sequence
\begin{equation}\label{Equ:HNLLB}
0\to L_{i-1}\to L_i\to B_i\to 0
\end{equation}
where $B_i \cong K_{i-1}$ by the octahedral axiom.  Taking cohomology sheaves shows that we have a chain 
\[
H^{-1}(L_1)\subset H^{-1}(L_2)\subset H^{-1}(L_3)\subset...\subset H^{-1}(E)
\]
which stabilizes after a finite number of terms. Hence by omitting finite number of terms, we can also assume that $H^{-1}(L_i)\simeq H^{-1}(L_{i+1})$ for all $i$.
Then 
the exact sequence \ref{Equ:4termexactseq} becomes 
\begin{equation}
\label{Equ:Qses}
0\to Q\to H^{-1}(E_i)\to H^0(L_i)\to 0,
\end{equation}
where $Q$ is the image of $f$ and is independent of $i$. We also have the following exact sequence of cohomology sheaves from \eqref{Equ:HNLLB}.
\[
0\to H^{-1}(B_i)\to H^0(L_{i-1})\to H^0(L_i)\to H^0(B_i)\to 0.
\]

Since $\Im(Z_H(L_i))=0$, we have $\Im(Z_H(H^0(L_i)))=0$, hence $H^0(L_i)$ lies in the extension closure of zero dimensional sheaves and sheaves in $\langle \cO_\Theta(j)|j>-2\rangle$. 
Since $L_{i-1}$, $L_i$ and $B_i$ are all in $\B^0_{H, -2}$, it follows that $\Im (Z_H (B_i))=0$ for all $i \geq 0$ as well, and we have 
\begin{itemize}
\item[(a)] $H^{-1}(B_i)=0$ or $\ch_2(H^{-1}(B_i))< 0$, 
\item[(b)] $\ch_2(H^0(L_{i-1}))\geq 0$, 
\item[(c)] $\ch_2(H^0(L_{i}))\geq 0$, 
\item[(d)] $\ch_2(H^0(B_i))\geq 0$.
\end{itemize}

Hence $\ch_2(H^0(L_i))\geq \ch_2(H^0(L_{i-1}))$.
Then we must have $\ch_2(H^0(L_i))=\ch_2(H^0(L_{i-1}))$ for $i\gg 0$. If not, by the short exact sequence \eqref{Equ:Qses} we have $\ch_2(H^{-1}(E_i))$ has no upper bound. Then $-\ch_2(E_i)>0$ for $i\gg 0$, which contradicts $Z_H$ being a weak stability function on $\Bc^0_{H,-2}$ (Proposition \ref{prop:weakstabfuncbefore}). 

Now that we know $\ch_2(H^0(L_i))$ is constant for $i \gg 0$, by (a) above we have  $H^{-1}(B_i)=0$ for $i\gg 0$ and
we have 
\[
0\to H^0(L_{i-1})\to H^0(L_i)\to H^0(B_i)\to 0.
\]
Then either $H^0(B_i)\in\langle \cO_\Theta(-1)\rangle$ or $H^0(B_i)=0$. Recall that  we have $B_{i+1}\simeq K_i$. Thus we have for $i\gg 0$, $K_i\in\langle \cO_\Theta(-1)\rangle$ or $K_i=0$.
Applying $\Hom(\cO_\Theta(-1), \_)$ to the short exact sequence \eqref{Equ:HNKEE} and recalling $\Ext^1(\OO_\Theta (-1), \OO_\Theta (-1))=0$ from Lemma \ref{lem:AG51-174-1}, if $K_i\neq 0$ then we have
\[
\dim\Hom(\cO_\Theta(-1), E_i)>\dim\Hom(\cO_\Theta(-1), E_{i+1}).
\]
Hence $K_i$ must vanish for $i \gg 0$, i.e.\ $E_i\simeq E_{i+1}$ for $i\gg 0$.  
This proves the HN property for $\sigma_b^L$.
%
%
\end{proof}

\begin{rem}\label{rem:AG52-17-2}
Since $\Bc_{H,-2} ^0 \cap \kernel (Z_H)$ is a Serre subcategory of $\Bc_{H,-2}^0$, and every object of $\Bc_{H,-2}^0 \cap \kernel (Z_H)$ is a direct sum of $\OO_X[1]$ and $\OO_\Theta (-1)$ which have different phases with respect to $\sigma^L_b$,  the only $\sigma^L_b$-semistable objects in $\Bc_{H,-2}^0 \cap \kernel (Z_H)$ are direct sums of $\OO_X[1]$ itself and direct sums of $\OO_\Theta (-1)$ itself.

\end{rem}
	
We denote the slicing of $\sigma_b^L$ by $\cP_b$. The next two results shows that given $\phi\in \mathbb{R}$, every object $E\in \cP_b(\phi)$ has a Jordan-H\"{o}lder filtration. 

\begin{lem} (Lemma 4.4 in \cite{SCK3})
	\label{Lem:Bri4.4}
 Let $\alpha$, $\beta$ be real numbers such that $0<\beta-\alpha<1-2\epsilon$ for some $0<\epsilon<\frac{1}{8}$.
	The thin subcategory $\cP_b((\alpha, \beta))$ is of finite length.
\end{lem}
\begin{proof}
	We follow the argument of Lemma 4.4 in \cite{SCK3}.
	Let $\phi=\frac{\alpha+\beta}{2}$. Then $\cP_b((\alpha, \beta))\subset \cP_b((\phi-\frac{1}{2}, \phi+\frac{1}{2}))$. Define a function 
	\begin{equation*}
		f(E)=\Re(\exp(-i\pi\phi)Z_H(E)).
	\end{equation*}
	Then for any $E\in \cP_b((\alpha, \beta))$, we have $f(E)>0$ or $E\in \ker(Z_H)$. Given a strict short exact sequence 
	\begin{equation*}
		0\rightarrow G\rightarrow E\rightarrow F\rightarrow 0
	\end{equation*}
	in $\cP_b((\alpha, \beta))$, we have $f(E)=f(G)+f(F)$. We have
	the following two cases:
	
	(i) the value of $f$ decreases when taking a subobject, 
	
	(ii) the value of $f$ does not change when taking a subobject and the corresponding quotient object is in $\ker(Z_H)$. 
	
	Assume that $\cO_X[1]\in \cP_b((\alpha, \beta))$ and/or $\cO_\Theta(-1)\in \cP_b((\alpha, \beta))$, then in case (ii) either the dimension of $\Hom(\_, \cO_X[1])$ or the dimension of $\Hom(\_, \cO_\Theta(-1))$ decreases when taking a subobject. Since $Z_H$ is discrete, any chain in $\cP_b((\alpha, \beta))$ of the form
	\begin{equation*}
		E_0\subset E_1\subset ...\subset E_i\subset...\subset E,
	\end{equation*}
	or of the form 
	\begin{equation*}
		...\subset E_{j+1}\subset ...\subset E_1\subset E
	\end{equation*} 
	must terminate. 
\end{proof}
\begin{prop}
	\label{Prop:JHfil}
 Consider the weak stability condition $\sigma_b^L$. 
	Any $E\in \cP_b(\phi)$ has a finite Jordan-H\"{o}lder filtration into stable factors.
\end{prop}
\begin{proof}
	WLOG, we assume that $0<\phi\leq 1$.
	If $E$ is $\sigma_b^L$- stable, then we are done. From now on we assume that $E$ is not $\sigma_b^L$-stable. 
 
 Step 1. We show that there exists a subobject $A_1$ of $E$ in $\B^0_{H, -2}$, such that $A_1$ is $\sigma_b^L$-stable.
 
Since $E$ is not stable, there exists a short exact sequence 
	\begin{equation*}
		0\rightarrow E_1\rightarrow E\rightarrow F_1\rightarrow 0
	\end{equation*}
	in $\B^0_{H, -2}$ with $\phi_b(E_1)=\phi_b(E)=\phi_b(F_1)=\phi$. 
	If $E_1$ is stable, then the claim is true. If not, we first prove the following Claim.
	
	Claim: for any 
	short exact sequence in $\B^0_{H, -2}$:
	\begin{equation}
		\label{Equ:JHfil1}
		0\rightarrow E_1^{(1)}\rightarrow E_1\rightarrow E_1'\rightarrow 0
	\end{equation}
	with $\phi_b(E_1^{(1)})\geq\phi_b(E_1')$, we have $\phi_b(E_1^{(1)})=\phi$. 
	
	Note that by assumption we have $\phi_b(E_1^{(1)})\geq\phi$. Composing with the map $E_1\rightarrow E$, we form a short exact sequence in $\B^0_{H, -2}$:
	$$0\rightarrow E_1^{(1)}\rightarrow E\rightarrow F_1^{(1)}\rightarrow 0.$$ 
	
	By the diagram 
	\begin{equation}
		\label{Equ:Dagger}
		\begin{tikzcd}
			& &  & 0\ar{d} &\\
			& 0\ar{d}&  &E_1'\ar{d} &  \\
			0\ar{r} & E_1^{(1)}\ar{r}\ar{d}&E \ar{r} \ar[equal]{d}&F_1^{(1)}\ar{r}\ar{d} &0\\
			0\ar{r}& E_1\ar{r}\ar{d}& E \ar{r} & F_1\ar{r}& 0\\
			& E_1'\ar{d}& & &\\
			& 0 & & &
		\end{tikzcd}
	\end{equation}
	we have $\phi_b(E_1^{(1)})\geq \phi_b(F_1^{(1)})$. Since $E$ is $\sigma^L_b$-semistable, we have $$\phi_b(E_1^{(1)})=\phi_b(F_1^{(1)})=\phi$$
 which proves the claim.
 
	We have either $E_1\in \cP_b(\phi)$, or $E_1$ is not semistable and in which case  there exists a short exact sequence in the form of  \eqref{Equ:JHfil1} with 
	$$\phi=\phi_b(E_1^{(1)})=\phi_b(E_1)>\phi_b(E_1').$$
	This implies that $E_1'\in \ker(Z_H)\cap\B^0_{H, -2}$. Then if $E_1\notin \cP_b(\phi)$, we have either $$\text{dim}(\Hom(E_1^{(1)}, \cO_X[1]))<\text{dim}(\Hom(E_1, \cO_X[1])),$$
	or $$\text{dim}(\Hom(E_1^{(1)}, \cO_\Theta(-1)))<\text{dim}(\Hom(E_1, \cO_\Theta(-1))).$$
	Replace $E_1$ by $E_1^{(1)}$ and continue this process, we see that there exists $m_1$ such that $E_1^{(m_1)}\in \cP_b(\phi)$. 
	
	If $E_1^{(m_1)}$ is stable, then we are done. If not, we replace $E$ by $E_1^{(m_1)}$ and repeat the whole process for $E_1^{(m_1)}$. 
	Then there exists a subobject of $E_1^{(m_1)}$ which is in $\cP_b(\phi)$, we denote this object by $E_1^{(m_2)}$.

	By Lemma \ref{Lem:Bri4.4}, The chain 
	$$...\subset E_1^{(m_i)}\subset...E_1^{(m_2)}\subset E_1^{(m_1)}\subset E$$
	in $\cP_b(\phi)$ must terminate. Then there exists $n_1$ such that  $E_1^{(n_1)}$ is stable. Denote this object by $A_1$.

 Step 2. We show that $E$ has a JH filtration. 

  Consider the short exact sequence in $\B^0_{H, -2}$:
 \begin{equation}
 \label{Equ:AEB}
 0\to A_1\to E\to B_1\to 0
 \end{equation}
 
 If $B_1\in \ker(Z_H)\cap \B^0_{H, -2}$, then by further filtering $B_1$ into stable factors, we have a JH filtration. If not, we have $\phi_b(B_1)=\phi_b(E)$.
Then a similar diagram as \eqref{Equ:Dagger} shows that for any 
	short exact sequence 
	\begin{equation}
		\label{Equ:JHfil2}
		0\rightarrow B'_1\rightarrow B_1\rightarrow B_1^{(1)}\rightarrow 0
	\end{equation}
	with $\phi_b(B'_1)\geq\phi_b(B_1^{( 1)})$, we have $\phi_b(B_1^{( 1)})=\phi$. We have either $B_1\in \cP_b(\phi)$, or there exists a short exact sequence in the form of  \eqref{Equ:JHfil2} such that $B'_1\in \ker(Z_H)\cap\B^0_{H, -2}$ and $\phi_b(B'_1)>\phi_b(B_1^{(1)})$. 
 Similar to the case of subobject, by comparing the dimension of $\Hom(\cO_X[1], \_)$ and $\Hom(\cO_\Theta(-1), \_)$, we also have $B_1^{(n_1)}\in \cP_b(\phi)$ for some $n_1$. 

 We abuse notation and denote the kernel of $B_1\to B_1^{(n_1)}$ by $B'_1$. Then $B'_1$ admits a filtration 
 \[
 0=B'_{1, 0}\subset B'_{1, 1}\subset B'_{1, 2}\subset...\subset B'_{1, l_1-1}\subset B'_{1, l_1}= B'_1
 \]
such that each of the quotient factors is stable in $\ker(Z_H)\cap\B^0_{H, -2}$. Denote the cokernel of $B'_{1, i}\to B_1$ by $B^{(n_1)}_{1, i}$, the kernel of $E\to B^{(n_1)}_{1, i}$ by $A^{(n_1)}_{1, i}$, and the kernel of $E\to B_1^{(n_1)}$ by $A_1^{(n_1)}$. Combining the short exact sequence \ref{Equ:AEB}, we have a set of short exact sequences 
\[
0\to A_1\to A^{(n_1)}_{1, i}\to B'_{1, i}\to 0
\]
 Hence we may further refine the map $A_1\to E$ to a filtration 
 \[
 A_1\subset A^{(n_1)}_{1, 1}\subset A^{(n_1)}_{1, 2}\subset...A^{(n_1)}_{1, l_1}=A_1^{(n_1)}\subset E
 \]
 with the last quotient factor isomorphic to $B_1^{(n_1)}$, and each of the previous quotient factors in $\ker(Z_H)\cap \B^0_{H, -2}$ and stable.

	If $B_1^{(n_1)}$ is stable, then we have a JH filtration. If not, 
	 we replace $E$ by $B_1^{(n_1)}$.  
 Then there exists a subobject of $B_1^{(n_1)}$ which is stable. We denote the cokernel by $B_2$. Then there exist 
 \[
 0\to B'_2\to B_2\to B_2^{(n_2)}\to 0
 \]
  where $B'_2\in \ker(Z_H)\cap B^0_{H, -2}$, and $B_2^{(n_2)}\in \cP_b(\phi)$. We can further refine the filtration such that $B'_2$ are stable objects in $\ker(Z_H)\cap B^0_{H, -2}$.  Continue this process and we denote the kernel of $E\to B_i^{(n_i)}$ by $A_i^{(n_i)}$.

Note that there exists $\alpha$, $\beta$ such that the interval $(\alpha, \beta)$ of length $<1$, and $\phi$, $\frac{1}{2}$ and $1$ are all in $(\alpha, \beta)$.
	 Then
\[A_1\subset A^{(n_1)}_{1, 1}\subset...\subset A_1^{(n_1)}\subset A_2\subset A^{(n_2)}_{2, 1}\subset...\subset A_2^{(n_2)}\subset A_3\subset...\subset E
 \]
 is a chain in $\cP_b((\alpha, \beta))$. 
By Lemma \ref{Lem:Bri4.4}, this chain must terminate. This is a Jordan-H\"{o}lder filtration of $E$.
\end{proof}

\begin{rem}\label{sec:weakstaboriginrmklast}
Note that Proposition \ref{Prop:JHfil} shows that any object $E\in \cP_b(\phi)$ has a JH filtration whose JH factors are in $\cP((0, 1])[k]$ for some $k\in \mathbb{Z}$. It is also worth pointing out that the $\sigma_b^L$-stable objects may not be stable in a general heart $\cP_b((a, a+1])$ for $a\notin \mathbb{Z}$.
\end{rem}

	\paragraph[Weak stability conditions on the $V$-axis] \label{sec:weakstabilityconditiontwostart}  Taking $D_\omega\to 0$ in the central charge formula \eqref{Equ:CentralChargeVD}, we obtain a central charge of the form $$Z_{V_\omega, H}(E)=-\ch_2(E)+V_\omega \ch_0(E)+iH\ch_1(E),$$
where $H=\Theta+ef$. Since $\omega$ does not appear in the formula, we omit $\omega$ in the notation.
\begin{prop}
	\label{Prop:ConstHeartV}
	We have $Z_{V, H}(\B^0_{H, k})\in \mathbb{H}_0$ for $k=-1, -2$. 
Furthermore, we have	

	when $k=-1$,
	$$\kernel (Z_{V, H})\cap \B^0_{H, -1}=\langle\cO_\Theta(-1)[1]\rangle,$$
	when $k=-2$,
	$$\kernel (Z_{V, H})\cap \B^0_{H, -2}= \langle \cO_\Theta(-1)\rangle.$$
\end{prop}
\begin{proof}
	For any object $E\in \B^0_{H, k}$, there is a short exact sequence 
	$$0\rightarrow F\rightarrow E\rightarrow T\rightarrow 0,$$
	where $F\in \F^0_{H, k}[1]$ and $T\in \T^0_{H, k}$. For an object $T\in\T^0_{H, k}$, $T$ fit into a short exact sequence 
	$$0\rightarrow T_1\rightarrow T\rightarrow T_2\rightarrow 0$$
	where $T_1\in \F^0_{H}[1]$ and $T_2\in \T^0_{H}$. 
 
 Claim 1. $Z_{V, H}(T_1)\in \mathbb{H}$. 

 Since $T_1\in \F^0_H[1]$, we have $T_1[-1]$ is a torsion free sheaf with $\mu_H(\mathrm{HN}_i(T_1[-1]))\leq 0$ for all $i$.
 If $J$ is a $\mu_H$-semistable coherent sheaf with $\mu_H(J)<0$, then $\Im Z_{V, H}(J[1])>0$. If $J$ is a $\mu_H$-stable torsion-free sheaf with $\mu_H(J)=0$, then by Corollary \ref{cor:K3surfstrongerBGdeform}, we have 
	\begin{equation*}
		\begin{split}
		\Re(Z_{V, H}(J[1]))&=\ch_2(J)-V\ch_0(J)\\
		&\leq \frac{\ch_1(J)^2}{2\ch_0(J)}-V\ch_0(J)\\
		&< 0
		\end{split}
	\end{equation*}
Hence $Z_{V, H}(T_1)\in \mathbb{H}$. 

Claim 2. $Z_{V, H}(T_2)\in \mathbb{H}_0$.
If $J$ is a $\mu_H$-semistable coherent sheaf with $\mu_H(J)>0$, then $Z_{V, H}(J)\in \mathbb{H}$. It is enough to consider if $T_2$ is torsion and $\mu_H(T_2)=0$. It is enough to assume that $T_2$ is a pure sheaf supported on dimension $1$. Then by Lemma \ref{lem:AG51-114}, we have $\ch_1(T_2)=n\Theta$, hence $T\in\langle \cO_\Theta(m)\rangle$. Since $(\T_{H, k}, \F_{H, k})$ defines a torsion pair on $\A^0_H$ and $T_2$ is an $\A^0_H$-quotient of $T$ which lies in $\Tc_{H,k}$, we have $T_2 \in \Tc_{H,k}$ as well and hence $T\in \langle\cO_\Theta(m)|m>k\rangle$. 

Since $$\ch_2(\cO_\Theta(m))=m+1,$$ we have for $m>k$, $$Z(\cO_\Theta(m))=-\ch_2(\cO_\Theta(m))<-k-1\in \mathbb{R}_{\leq 0}.$$ 
This implies that $Z_{V, H}(T_2)\in \mathbb{H}_0$.

We are left to consider $F$. 
For $m\leq k$ $$Z(\cO_\Theta(m)[1])=\ch_2(\cO_\Theta(m))=m+1\leq k+1\in \mathbb{R}_{\leq 0}.$$ 
Hence $Z_{V, H}(F)\in\mathbb{H}_0$.

 From the computation, we also have for $k=-1$
$$\kernel (Z_{V, H})\cap \B^0_{H, k}= \langle\cO_\Theta(-1)[1]\rangle,$$ 
and for $k=-2$, $$\kernel (Z_{V, H})\cap \B^0_{H, k}= \langle\cO_\Theta(-1)\rangle.$$ 
\end{proof}

We define the phases $\{\phi_{V, H}(K)\}_{K\in \ker(Z_{V, H})\cap \B^0_{H, k}}$ by taking the limit of the phases of $\sigma_{V_\omega, D_\omega}(K)$ as $D_\omega\to 0$.

Since
$$Z_{V_\omega, D_\omega}(\cO_\Theta(-1))=iD_\omega,$$
we have $$\lim\limits_{D_\omega\rightarrow 0}\phi_{V_\omega, D_\omega}(\cO_\Theta(-1))=\frac{1}{2}.$$
Hence for $V\in\mathbb{R}_{>0}$, we define $\phi_{V, H}(\OO_\Theta (-1))=\frac{1}{2}$, which implies $\phi_{V, H}(\cO_\Theta(-1)[1])=\frac{3}{2}$.

We see that it is  only when $k=-2$ that  we have a chance of having  $0<\phi_{V, H}(K)\leq 1$ for all $K \in \kernel (Z_{V,H}) \cap \Bc^0_{H,k}$. Since $\ker(Z_{V, H})\cap \B^0_{H, -2}=\langle \cO_\Theta(-1)\rangle$, for any $K\in \ker(Z_{V, H})\cap \B^0_{H, -2}$, we have
\[
\lim\limits_{D_\omega\rightarrow 0}\phi_{V_\omega, D_\omega}(K)=\frac{1}{2}.
\]
As a result, for $V\in\mathbb{R}_{>0}$, $K\in\ker(Z_{V, H})\cap \B^0_{H, -2}$ we define $\phi_{V, H}(K)=\frac{1}{2}$. 
 Then the weak see-saw property follows automatically. To summarize, the triple 
$$\sigma_{V, H}=(Z_{V, H}, \B^0_{H, -2}, \{\phi_{V, H}(K)\}_{K\in\ker(Z_{V, H})\cap\B^0_{H, -2}})$$
defines a weak stability function. 

The HN property of $\sigma_{V, H}$ can be shown exactly the same way as Proposition \ref{Prop:HNorigin}, and so $\sigma_{V,H}$ is a weak stability condition in the sense of Definition \ref{Def:Weak}. We denote the slicing of $\sigma_{V, H}$ by $\cP_{V, H}$.

\begin{rem}\label{rmk:weakstabilityconditiontwoend}
From the computation in Lemma \ref{lem:AG51-174-1}, we know that $\Ext^1(\cO_\Theta(-1), \cO_\Theta(-1))=0$. Hence following the same proof of Propostiion \ref{Prop:JHfil}, we have any object $E\in \cP_{V, H}(\phi)$ has a JH filtration whose stable factors are either $\sigma_{V, H}$-stable objects in $\cP_{V, H}(\phi)$ or stable objects in $\ker(Z_{V, H})\cap \B^0_{H, -2}$.
\end{rem}

\paragraph[Weak stability conditions on the $D$-axis] \label{sec:weakstabilityconditionthreestart}  Taking $V_\omega\to 0$ in the central charge formula \eqref{Equ:CentralChargeVD}, we obtain a central charge of the form 
$$Z_{D_\omega}(E)=-\ch_2(E)+i\omega \ch_1(E)$$
where $\omega=\Theta+(D_\omega+e)f$. If $\omega$ is clear from the context, we omit $\omega$ in the notation. 
\begin{prop}
	\label{Prop:ConstHeartD}
	Suppose $\omega$ is ample.  We have $Z_{D}(\Coh^{\omega, 0})\in\mathbb{H}_0$.
 Furthermore, $\kernel (Z_{D})\cap \Coh^{\omega, 0}=\langle\cO_X[1]\rangle$.
	\end{prop}
\begin{proof}
	We only need to show $Z_D(E[1])\in \mathbb{H}_0$ for $E$ a $\mu_{\omega}$-semistable torsion-free sheaf with $\mu_\omega(E)=0$. In this case,
	\begin{equation*}
		Z_{D}(E[1])=\ch_2(E)\leq \frac{\ch_1(E)^2}{2\ch_0(E)}\in\mathbb{R}_{\leq 0}.
	\end{equation*}
Hence $Z_D(E[1])\in \mathbb{H}_0$. 

Furthermore, if $E\in \ker(Z_{D})\cap \Coh^{\omega, 0}$, then we have $E=E'[1]$ for $E'\in \Coh (X)$. 
By the generalized Bogomolov-Gieseker inequality on K3 surfaces, (e.g.\ see \cite[Section 6]{ABL} or \eqref{eq:K3surfstrongerBG}), 
if $F$ is a $\mu_\omega$-stable object, then 
\begin{equation}
\label{Equ:GenBG}
\ch_2(F)\leq \frac{\ch_1(F)^2}{2\ch_0(F)}-\ch_0(F)+\frac{1}{\ch_0(F)}.
\end{equation}
If furthermore $F[1]\in \kernel (Z_{D})\cap \Coh^{\omega, 0}$, then $\ch_0(F)=1$. This implies that $F$ is an ideal sheaf. Since $\ch_1(F)\cdot\omega=0$ and $\ch_2(F)=0$, we have $F\simeq \cO_X$. Hence $E'\in\langle\cO_X\rangle$.
	\end{proof}
	
We define the phases $\{\phi_D(K)\}_{K\in \kernel (Z_D)\cap \Coh^{\omega, 0}}$ by taking the limit of the phases of $Z_{V_\omega, D_\omega}(K)$ as $V_\omega\to 0$.
We have 
$$Z_{V_\omega, D_\omega}(\cO_X[1])=-V_\omega.$$
Hence $$\lim\limits_{V_\omega\rightarrow 0}\phi_{V_\omega, D_\omega}(\cO_X[1])=1.$$ For $D\in\mathbb{R}_{>0}$, we define $\phi_{D}(K)=1$ for any $K \in \kernel (Z_D) \cap \Coh^{\omega, 0} = \langle \OO_X[1]\rangle$.

\begin{prop}\label{prop:HNDaxis-Q}
	Given $D\in\mathbb{Q}_{>0}$, the triple
	\begin{equation*}
	\sigma_D=(Z_{D}, \Coh^{\omega, 0}, \{ \phi_D(K) \}_{K \in \kernel (Z_D) \cap \Coh^{\omega,0}}   )
	    \end{equation*}
 satisfies the HN property, hence defines a weak stability condition.
\end{prop}

\begin{proof}
	The argument follows the proof of Lemma 2.18 in \cite{piyaratne2015moduli}.
		
Since $\omega$ is ample, it is well known that the abelian category $\Coh^{\omega, 0}$ is noetherian (e.g.\  see  \cite[Lemma 6.17]{MSlec}. As a result, we only need to check condition (i) in  Proposition \ref{Prop:HNfil}. Assume we have such a chain of subobjects in $\Coh^{\omega, 0}$ as in condition (i). Consider the short exact sequence in $\Coh^{\omega, 0}$:
		$$0\rightarrow E_{i+1}\rightarrow E_i\rightarrow F_i\rightarrow 0.$$
		Since $\Im(Z_D(F_i))\geq 0$, we have $\Im(Z_D(E_{i}))\geq \Im(Z_D(E_{i+1}))$. 
		The assumption $D_\omega\in \mathbb{Q}_{>0}$ implies that $\Im(Z_D)$ is discrete. Hence $\Im(Z_D(E_i))$ is constant for $i\gg 0$. Then $\Im(Z_D(F_i))=0$ for $i>n$ for some $n$. 
		
		Let $i>n$.  If $F_i\notin \ker(Z_D)\cap \Coh^{\omega, 0}$, then $\phi_D(F_i)=1$, contradicting  $\phi_D(E_{i+1})>\phi_D(F_i)$.
		
		So we must have $F_i\in \ker(Z_D)\cap \Coh^{\omega, 0}=\langle \cO_X[1]\rangle$ for $i>n$. 
Since $\Ext^1(\cO_X[1], \cO_X[1])=0$,
 the dimension of $\Hom(E_{i}, \cO_X[1])$ decreases as $i$ increases. This implies that $F_i=0$ for $i$ large enough. This proves the HN property for $\sigma_D$ when $D\in \mathbb{Q}_{>0}$.
\end{proof}

 \paragraph[Weak stability condition after the relative Fourier-Mukai transform] \label{sec:weakstabilityconditionfourstart}  Recall that on a smooth projective surface $X$, we usually denote the central charge of a  standard Bridgeland stability condition as
  \begin{align}
  Z_{\omega, B}&=-\int_X e^{-(B+i\omega)}\ch(E) = -\int_X e^{-i\omega}\ch^B(E) \notag\\
  &= -\ch_2^B(E) + V_\omega \ch_0^B(E) + iR_\omega (\Theta + (D_\omega + e)f) \ch_1^B(E) \label{eq:Zwbcceq-RDV}
\end{align}
 for $\mathbb{R}$-divisors $\omega$ and $B$.  To construct a heart that pairs with $Z_{\omega, B}$ to form a stability condition on $D^b(X)$, we usually consider  the torsion pair $(\Tc_{\omega, B}, \Fc_{\omega, B})$ where
 \begin{align*}
     \Tc_{\omega, B}&=  \langle E \in \Coh (X) : \mu_{\omega, B, \mathrm{min}}(E) > 0 \rangle \\
     \Fc_{\omega, B}&=  \langle E \in \Coh (X) : \mu_{\omega, B, \mathrm{max}}(E) \leq 0 \rangle.
 \end{align*}
 The heart $\Ac_{\omega, B} = \langle \Fc_{\omega, B}[1], \Tc_{\omega, B}\rangle$ then pairs with $Z_{\omega, B}$ to form a Bridgeland stability condition on $D^b(X)$.
 
Now let us return to the case of $X$ being a Weierstra{\ss} elliptic surface. Given $\omega$ and $B$, when attempting to solve the central charge equation 
 \begin{equation*}
  Z_{\omega', B'}(\Phi (E)) = T Z_{\omega, B}(E) \text{\quad for all $E \in D^b(X)$} 
 \end{equation*}
 for some $T \in \mathrm{GL}^+\!(2,\mathbb{R})$ and $\omega', B'$, it is easy to see from the solution in \cite[8.5]{Lo20}, that when $\omega^2$ is small, the above equation admits a solution where $V_{\omega'}$, i.e.\ the coefficient of $\ch_0$ in $Z_{\omega', B'}$, is forced to be negative.  This prompts us to consider the central charge with Todd class on a Weierstra{\ss} elliptic K3 surface $X$:
 \begin{align}
Z^{td}_{\omega', B'} (E) &= (e^{\omega'+B'}, \nu(E))
			=-\int e^{-i\omega'} \ch^{B'}(E)\sqrt{\mathrm{td}(X)} \notag \\
   &= -\ch_2^{B'}(E) + (\tfrac{(\omega')^2}{2}-1)\ch_0^{B'}(E) + i\omega' \ch_1^B(E) \notag \\
   &= -\ch_2^{B'}(E) + (V_{\omega'} -1)\ch_0^{B'}(E) + i\omega' \ch_1^{B'}(E). \label{Equ:CenCharTd}
 \end{align}

In Appendix \ref{sec:app-solvingcce}, we show that when we impose the relations

	\begin{equation}
		\begin{split}
			\label{Equ:RDVNoTd_Td}
			D_{\omega'}&=\frac{1}{2}(R_{\omega}^2+R_{B}^2)(2D_{\omega}+e)\\
			&=V_{\omega}+\frac{1}{2}R_{B}^2(2D_{\omega}+e).\\
			R_{B'}&=-\frac{R_{B}(2D_{\omega}+e)}{2D_{\omega'}+e}.\\
			V_{\omega'}
			&=D_{\omega}-\frac{1}{2}\frac{R_{B}^2(2D_{\omega}+e)^2}{2D_{\omega'}+e}+1\\
			R_{B'}D_{B'}&=-R_{B'}+R_BD_B+R_B-1
		\end{split}
	\end{equation} 
there is an appropriate $T\in \mathrm{GL}^+\!(2,\mathbb{R})$ such that the slightly different central charge equation
 \begin{equation}\label{eq:cce-var1}
Z_{\omega', B'}^{td}(\Phi (E)) = T Z_{\omega, B}(E) \text{\quad for all $E \in D^b(X)$} 
 \end{equation}
holds.

\paragraph We are interested in 
solving the equation \eqref{eq:cce-var1} on an elliptic K3 surface for the case  $B=-\alpha$ where $\alpha := c_1(L)=\Theta + (D_\alpha + e)f$.  In this case, we have $R_{B}=-R_\alpha=-1$ and $D_{B}=D_\alpha$, and  \eqref{Equ:RDVNoTd_Td} simplifies to
	\begin{equation}
		\label{Equ:PsiZ}
		\begin{split}
			D_{\omega'} &=V_{\omega}+D_{\omega}+1\\
			R_{B'} &=\frac{D_{\omega}+1}{D_{\omega}+V_{\omega}+2}\\
			V_{\omega'} &=\frac{D_{\omega}V_{\omega}-1}{D_{\omega}+V_{\omega}+2}+1\\
			R_{B'}D_{B'} &=-\frac{D_{\omega}+1}{V_{\omega}+D_{\omega}+2}-(D_\alpha+2).
		\end{split}
	\end{equation}
	
Then equations in \eqref{Equ:PsiZ} define a map from
	$\mathbb{R}^4$ to $\mathbb{R}^4$. 
	We denote this map by $\Phi_Z$.

In  formula \eqref{eq:Zwbcceq-RDV} for the central charge $Z_{\omega, B}$, if we choose the parameters 
\[
D_\omega = 0, \,\,\, V_\omega = 0, \,\,\,  R_\omega = 1, \,\,\, B=-\alpha
\]
 then $Z_{\omega, B}$ reduces to $Z_H ((-)\otimes \OO_X(-B))$ where  $Z_H$ is as defined in \ref{section:weakstaborigin}.  With these choices, the relations \eqref{Equ:PsiZ} give
 \[
 D_{\omega'}=1,\,\,\, R_{B'}=\tfrac{1}{2},\,\,\, V_{\omega'}=\tfrac{1}{2}, \,\,\, R_{B'}D_{B'}=-D_\alpha -\tfrac{5}{2}
 \]
 which in turn give
 \begin{align*}
  \omega' &= R_{\omega'}(\Theta + (D_{\omega'}+e)f) = \tfrac{1}{2}(\Theta + 3f) \text{ where $R_{\omega'}=\sqrt{\frac{V_{\omega'}}{D_{\omega'}+\tfrac{e}{2}}}$} \\
  B' &= \tfrac{1}{2}(\Theta + (-2D_\alpha -3)f).
 \end{align*}
 In summary, with the RDV coordinates for $\omega, B, \omega', B'$ chosen as above, the equation 
 \begin{equation}\label{eq:AG52-47-1}
Z_{\omega', B'}^{td}(\Phi (E)) = T Z_H(E\otimes \OO_X(-B)) \text{\quad for all $E \in D^b(X)$} 
 \end{equation}
holds for the appropriate $T$.

In what follows, we will  write $\omega_0', B_0'$ to denote the specific $\omega', B'$ above, i.e.\
\begin{equation}
		\label{Equ:SpecialPoint}
		\begin{split}
			\omega'_0&=\tfrac{1}{2}(\Theta+3f)\\
			B'_0&=\tfrac{1}{2}(\Theta+(-2D_{\alpha}-3)f).
		\end{split}
	\end{equation}
We will   informally think of the left-hand side of \eqref{eq:AG52-47-1} as ``$\Phi_Z(Z_{H})$'' and denote it by $Z'_0:=Z^{td}_{\omega'_0, B'_0}$.  Also note that $\omega_0'$ is ample. 
	
	

\begin{prop}
\label{prop:weakstabf-afttrans}
Using the notations above, we have $Z'_0(\Coh^{\omega'_0, B'_0})\in \mathbb{H}_0$.
\end{prop}
\begin{proof}
Since $\omega_0'$ is ample, it is easy to see that if $E\in \Coh(X)$ is a torsion sheaf, or is a  slope semistable torsion-free sheaf with $\mu_{\omega'_0, B'_0}(E)>0$, then $Z'_0(E)\in \mathbb{H}$. Also if $E\in \Coh(X)$ is slope semistable with $\mu_{\omega'_0, B'_0}(E)<0$, then $Z'_0(E[1])\in \mathbb{H}$.  It remains to check that  when $E\in \Coh(X)$ is  a slope stable torsion-free sheaf with $\mu_{\omega'_0, B'_0}(E)=0$, \cblue{we have $Z_0'(E[1]) \in \mathbb{H}_0$, i.e. $\Re(Z'_0(E[1]))\leq 0$.} We have 
\begin{equation*}
\begin{split}
\Re(Z'_0(E[1]))&=\ch_2^{B'_0}(E)-(V_{\omega'_0}-1)\ch_0^{B'_0}(E)\\
&=\ch_2^{B'_0}(E)+\frac{1}{2}\ch_0^{B'_0}(E).
\end{split}
\end{equation*}

By the generalized BG inequality on K3 surfaces in Proposition \ref{prop:K3surfstrongerBG}, we have 
\[
\cblue{\Re(Z'_0(E[1]))}\leq \frac{(\ch_1^{B'_0}(E))^2}{2\ch_0(E)}-\frac{1}{2}\ch_0(E)+\frac{1}{\ch_0(E)}.
\]
By Hodge index theorem, we have $(\ch_1^{B'_0}(E))^2\leq 0$, hence $\Re(Z'_0(E[1]))< 0$ when $\ch_0(E)\geq 2$. We only need to consider the case when $\ch_0(E)=1$.
In this case we have 
\[
\cblue{\Re(Z'_0(E[1]))}\leq \frac{1}{2}(\ch_1^{B'_0}(E))^2+\frac{1}{2}.
\]
We have 
\begin{equation*}
\begin{split}
(\ch_1^{B'_0}(E))^2&=\ch_1(E)^2-2B'_0\ch_1(E)+(B'_0)^2\\
&=\ch_1(E)^2-(\Theta+(-2D_\alpha-3)f)\ch_1(E)+(-D_\alpha-2).
\end{split}
\end{equation*}
Hence $(\ch_1^{B'_0}(E))^2$ is an integer. If $(\ch_1^{B'_0}(E))^2=0$, Hodge index theorem implies that $\ch_1^{B'_0}(E)$ is numerically trivial. This implies that 
\[
\ch_1(E)=B'_0=\frac{1}{2}(\Theta+(-2D_\alpha-3)f),
\]
contradicting $E\in \Coh(X)$. Hence $(\ch_1^{B'_0}(E))^2\leq -1$, which implies that $\Re(Z'_0(E[1]))\leq 0$.
\end{proof}


We analyze the objects in the kernel of the central charge. 
 Let $K\in \ker(Z'_0)\cap \Coh^{\omega'_0, B'_0}$, then $K=E[1]$ for some sheaf $E\in \mathcal{F}_{\omega', B'}$. 


	\begin{prop}
		\label{Prop:L0L1}
		Let $E$ be a $\mu_{\omega', B'}$-stable sheaf with 
		$E[1]\in \ker(Z'_0)\cap \Coh^{\omega'_0, B'_0}$. 
		Then either $E\simeq \cO(-(D_\alpha+1)f)$ or $E\simeq \cO(\Theta-(D_\alpha+2)f)$.
	\end{prop}
 
	\begin{proof}
  From the proof of Proposition \ref{prop:weakstabf-afttrans}, we know that $\ch_0(E)=1$, and that equality holds for $E$ in the generalized BG inequality, giving us
  \[
\Re(Z'_0(E[1]))= \tfrac{1}{2}(\ch_1^{B'_0}(E))^2+\tfrac{1}{2},
\]
hence $\ch_2(E)=\frac{1}{2}\ch_1(E)^2$. 
Consider the short exact sequence 
\[
0\to E\to E^{**}\to Q\to 0.
\]
Since $E$ is $\mu_{\omega'_0, B'_0}$-stable of rank $1$, we have $Q$ is supported on dimension $0$. Then $\ch_2(E)=\frac{1}{2}\ch_1(E)^2$ implies that $E\simeq E^{**}$, and hence $E$ is a line bundle. 

Since $\Im(Z'_0(E))=0$ and $\Re(Z'_0(E))=0$, we have
	\begin{equation}
 \label{Equ:KerafterFM}
		\begin{split}
			&\ch_1(E)\cdot(\Theta+3f)-\frac{1}{2}(1+(-2D_\alpha-5)+2)=0\\
			&\frac{1}{2}\ch_1(E)^2-\ch_1(E)\cdot\frac{1}{2}(\Theta+(-2D_\alpha-3)f)+\frac{1}{4}(-2D_\alpha-4)+\frac{1}{2}=0.
		\end{split}
	\end{equation}
		Denote $\ch_1(E)\cdot\Theta$ by $x$, and $\ch_1(E)\cdot f$ by $y$.  
		Then  equation \eqref{Equ:KerafterFM} becomes
		\begin{equation}
			\label{Equ:KerNum}
			\begin{split}
				&x+3y=-D_\alpha-1\\
				&\ch_1(E)^2=-(2D_\alpha+6)y.
			\end{split}
		\end{equation}
		For any $a,b\in\mathbb{R}$ such that $a+b=D_\alpha+1$, we have $ (\ch_1(E)+(a\Theta+b f))\cdot \omega_0'=0$. Then the Hodge index theorem implies that
		\begin{equation}
			\label{Equ:HIT}
			(\ch_1(E)+(a\Theta+b f))^2\leq 0.
		\end{equation}
		Substituting $b=D_\alpha-a+1$, equation \eqref{Equ:HIT} implies that
		\begin{equation*}
			(2a+1)y\geq -a^2
		\end{equation*}
		for all $a\in\mathbb{R}$.
		Taking $a=0$, we have $y\geq 0$. Taking $a=-1$, we have $y\leq 1$.
		Since $y\in \mathbb{Z}$, we have $y=0$ or $y=1$.
		If $y=0$, then HIT taking equality forces $\ch_1(E)$ to be numerically equivalent to $-(D_\alpha+1)f$. Similarly, if $y=1$, then $\ch_1(E)$ is numerically equivalent to $\Theta-(D_\alpha+2)f$.
	\end{proof}
	
\paragraph \label{para:L0L1defs}	Define $$L_0:=\cO(-(D_\alpha+1)f),$$ and $$L_1:=\cO(\Theta-(D_\alpha+2)f).$$ 
Let $E$, $y$ be the same as in Proposition \ref{Prop:L0L1}. 
We denote the phase of $E$ in $\sigma'=(Z^{td}_{\omega', B'}, \Coh^{\omega', B'})$ by $\phi_{\sigma'}(E)$.
	
	
	Precomposing with the map $\Phi_Z$, we have 
	\begin{equation}
		\label{Equ:ImReNbd1}
		\begin{split}
			&\Im(Z_{\omega', B'}(E))=\frac{\sqrt{(D_{\omega}+1)(V_{\omega}+1)}}{D_{\omega}+V_{\omega}+2}(y(D_{\omega}+V_{\omega})-D_{\omega}).\\
			&\Re(Z_{\omega', B'}(E))=\frac{V_{\omega}-D_{\omega}}{D_{\omega}+V_{\omega}+2}y+\frac{D_{\omega}V_{\omega}+D_{\omega}}{D_{\omega}+V_{\omega}+2}.
		\end{split}
	\end{equation}
	For the rest of the section, we always consider $\sigma'$ as a function of $V_\omega$ and $D_\omega$.
 
	If $E\simeq L_0$, i.e. $y=0$, we have 
	\begin{equation}
		\label{Equ:ImReNbd2}
		\begin{split} &\Im(Z_{\omega', B'}(L_0))=-D_{\omega}\frac{\sqrt{(D_{\omega}+1)(V_{\omega}+1)}}{D_{\omega}+V_{\omega}+2}\\ 
			&\Re(Z_{\omega', B'}(L_0))=\frac{D_{\omega}V_{\omega}+D_{\omega}}{D_{\omega}+V_{\omega}+2}.
		\end{split}
	\end{equation}
	
	Since $D_\omega>0$, $V_\omega>0$, we have $\Im(Z_{\omega', B'}(L_0[1]))>0$, and since 
 \[
 -\frac{\Re(Z_{\omega', B'}(L_0[1]))}{\Im(Z_{\omega', B'}(L_0[1]))} = \frac{V_\omega + 1}{\sqrt{(D_\omega + 1)(V_\omega+1)}}
 \]
 it is easy to see that 
	$$\lim\limits_{D_\omega, V_\omega\to 0^+ }\phi_{\sigma'}(L_0[1])=\frac{3}{4}.$$
	
	If $E\simeq L_1$, i.e. $y=1$, we have 
	\begin{equation}
		\label{Equ:ImReNbd3}
		\begin{split}
			&\Im(Z_{\omega', B'}(L_1))=V_{\omega}\frac{\sqrt{(D_{\omega}+1)(V_{\omega}+1)}}{D_{\omega}+V_{\omega}+2}\\
			&\Re(Z_{\omega', B'}(L_1))=\frac{D_{\omega}V_{\omega}+V_{\omega}}{D_{\omega}+V_{\omega}+2}.
		\end{split}
	\end{equation}
	In this case we have $\Im(Z_{\omega', B'}(L_1))>0$, and 
 \[
 -\frac{\Re(Z_{\omega', B'}(L_1))}{\Im(Z_{\omega', B'}(L_1))} = -\frac{D_\omega + 1}{\sqrt{(D_\omega + 1)(V_\omega+1)}}
 \]
 and so 
 $$\lim\limits_{D_\omega, V_\omega\to 0^+}\phi_{\sigma'}(L_1)=\frac{1}{4}.$$
	
	From the computation, we see that if $K\in \ker(Z'_0)\cap \Coh^{\omega', B'}$ is taken to be $L_1[1]$, then $$\lim\limits_{D_\omega, V_\omega\to 0^+}\phi_{\sigma'}(K)=\frac{5}{4}.$$
 If we define
	\begin{equation}
		\label{Equ:Limitphase}
		\phi_{\sigma'_0}(K)=\lim\limits_{D_\omega, V_\omega\to 0^+} \phi_{\sigma'}(K),
	\end{equation}  
	the triple $$(Z'_0, \Coh^{\omega'_0, B'_0}, \{\phi_{\sigma'_0}(K)\}|_{K\in Ker(Z'_0)\cap \Coh^{\omega'_0, B'_0}})$$ does not satisfy condition (i) in Definition \ref{Def:Weak}. To define a weak stability condition, we construct a heart $\mathcal{B}$ such that $L_1\in \mathcal{B}$ and $Z'_0(\B)\in \mathbb{H}_0$. To simplify notation, for the rest of this section, we denote $\Coh^{\omega'_0, B'_0}$ by $\A$. Recall that $\A$ is constructed by tilting $\Coh(X)$ at a torsion pair $(\T_{\omega'_0, B'_0}, \F_{\omega'_0, B'_0})$. For the rest of this section, we denote this torsion pair by $(\T, \F)$. We construct $\B$ by tilting $\A$ at a torsion pair. 

	It is noted in \cite{piyaratne2015moduli} that the full subcategory $\A_{\ker(Z'_0)}:=\ker(Z'_0)\cap \A$ is an abelian subcategory of $\A$.
 	\begin{lem}
		\label{Lem:Kertorpair}
		Any object in $\A_{\ker(Z'_0)}$ is a direct sum of copies of $L_0[1]$ and $L_1[1]$.
	\end{lem}
	\begin{proof}

Let $K\in \A_{\ker(Z'_0)}$, we know that $K\simeq E[1]$ for some coherent sheaf $E\in \F_{\omega', B'}$. By Proposition \ref{Prop:L0L1}, we have $E$ has a filtration by copies of $L_0$ and $L_1$. 

We have $\Hom(L_1, L_0)=H^0(O(-\Theta+f))=0$. 
Also, Serre duality implies that $$\Hom (L_0, L_1) \cong H^2(O(-\Theta+f))\simeq H^0(O(\Theta-f))^*=0.$$ 
By Riemann-Roch, we have $$\Ext^1(L_1, L_0)=H^1(O(-\Theta+f))=0$$
		and 
		$$\Ext^1(L_0, L_1)=H^1(O(\Theta-f))=0.$$
	\end{proof}
	Define the subcategory $\T_\A\subset \A$ by 
	\begin{equation*}
		\T_\A=\langle L_1[1]\rangle.
	\end{equation*}

	Define its right orthogonal in $\A$ by 
	\begin{equation*}
		\F_\A=\{F\in \A|\Hom(L_1[1], F)=0\}.
	\end{equation*}
	
	\begin{prop}\label{prop:torsionpairTAFA}
		The pair $(\T_\A, \F_\A)$ defines a torsion pair on $\A$.
	\end{prop}
	\begin{proof}
		The argument is similar to the proof of Lemma 3.2 in \cite{tramel2017bridgeland}.
		
		We need to show that for any $E\in \A$, there exists an exact triangle 
		\begin{equation*}
			0\rightarrow T\rightarrow E\rightarrow F\rightarrow 0
		\end{equation*}
		with $T\in \T_\A$ and $F\in \F_\A$.
		Assume that $E\notin \F_\A$, then there exists $T\in \T_\A$ with a nonzero map to $E$.  Completing the triangle, we have 
		\begin{equation*}
			T\rightarrow E\rightarrow F\rightarrow T[1].
		\end{equation*}
		Since $H^i(T)=0$ for all $i\neq -1$, by long exact sequence of cohomology, we have 
		\begin{equation*}
			0\rightarrow H^{-2}(F)\rightarrow H^{-1}(T)\rightarrow H^{-1}(E)\rightarrow H^{-1}(F)\rightarrow 0.
		\end{equation*}
		Consider the map $H^{-1}(T)/H^{-2}(F)[1]\rightarrow E$ via the composition $$H^{-1}(T)/H^{-2}(F)[1]\rightarrow H^{-1}(E)[1]\rightarrow E.$$ Completing the triangle, we obtain the following exact triangle
		\begin{equation}
			\label{Equ:TorpairStep1}
			H^{-1}(T)/H^{-2}(F)[1]\rightarrow E\rightarrow F_1\rightarrow H^{-1}(T)/H^{-2}(F)[2].
		\end{equation}

  Claim 1. $H^{-1}(T)/H^{-2}(F)[1]\in\T_\A$.
		We have the following short exact sequence of sheaves:
		\begin{equation}
			\label{Equ:longexseq}
			0\rightarrow H^{-2}(F)\rightarrow H^{-1}(T)\rightarrow H^{-1}(T)/H^{-2}(F)\rightarrow 0.
		\end{equation}
		Since $H^{-1}(T)\in \F$, we have $H^{-2}(F)\in \F$. Also since $H^{-1}(T)/H^{-2}(F)$ is a subsheaf of $H^{-1}(E)$, we have $H^{-1}(T)/H^{-2}(F)\in \F$. Then the short exact sequence \eqref{Equ:longexseq} is a short exact sequence in $\A[-1]$.
  Then $H^{-1}(T)[1]\in \A_{\ker(Z'_0)}$ implies that $H^{-2}(F)[1]\in \A_{\ker(Z'_0)}$ and $H^{-1}(T)/H^{-2}(F)[1]\in\A_{\ker(Z'_0)}$. Hence we have equation \eqref{Equ:longexseq} is a short exact sequence in $\A_{\ker(Z'_0)}[-1]$. From the proof of Lemma \ref{Lem:Kertorpair}, $H^{-1}(T)[1]\in \T_\A$ implies that $H^{-1}(T)/H^{-2}(F)[1]\in \T_\A$.
		
		Claim 2. $F_1\in \A$. Indeed by long exact sequence, we have $H^0(F_1)\simeq H^0(E)\in \T$, and 
		\begin{equation*}
			0\rightarrow H^{-1}(T)/H^{-2}(F)\rightarrow H^{-1}(E)\rightarrow H^{-1}(F_1)\rightarrow 0.
		\end{equation*} 
		Let $G$ be any subsheaf of $H^{-1}(F_1)$, and let $R$ be the cokernel. Let $K$ be the kernel of the composition $H^{-1}(E)\rightarrow H^{-1}(F_1)\rightarrow R$. Then we have a short exact sequence of coherent sheaves:
		\begin{equation*}
			0\rightarrow H^{-1}(T)/H^{-2}(F)\rightarrow K\rightarrow G\rightarrow 0.
		\end{equation*}
		$H^{-1}(T)/H^{-2}(F)[1]\in \A_{\ker(Z'_0)}$ implies that $\mu_{\omega', B'}(K)=\mu_{\omega', B'}(G)$. Since $K$ is a subsheaf of $H^{-1}(E)$, $\mu_{\omega', B'}(K)\leq 0$. Hence we have $\mu_{\omega_1, B_1}(G)\leq 0$, which implies that $H^{-1}(F_1)\in \F$.
		
		Hence the exact triangle \eqref{Equ:TorpairStep1} is a short exact sequence in $\A$, and $H^{-1}(T)/H^{-2}(F)[1]\in \T_\A$. If $F_1\in \F_\A$, then we are done. If not, we continue the process for $F_1$. Thus we obtain a sequence of quotients

\begin{equation}
\label{Equ:seqquotient}
F_1\twoheadrightarrow F_2\twoheadrightarrow F_3\twoheadrightarrow...
\end{equation}
		in $\A$, where the kernels $K_i$'s lie in $\T_\A$. Applying $\Hom(L_1[1],-)$ to the short exact sequences
		\begin{equation*}
			0\rightarrow K_i\rightarrow F_i\rightarrow F_{i+1}\rightarrow 0,
		\end{equation*}
		 we have 
		\begin{equation*}
			0\rightarrow \Hom(L_1[1], K_i)\rightarrow \Hom(L_1[1], F_i)\rightarrow \Hom(L_1[1], F_{i+1})\rightarrow \Ext^1(L_1[1], K_i)=0.
		\end{equation*}
		Hence the dimension of $\Hom (L_1[1], F_i)$ decreases as $i$ increases. So we have the sequence \eqref{Equ:seqquotient} must stabilize after a finite number of steps.
	\end{proof}

 \begin{defn}\label{def:BtiltofA}
 We define an abelian category $\B$ by tilting $\A[-1]$ at the torion pair $(\T_\A[-1], \F_\A[-1])$. Equivalently, $\B=\langle\F_\A, \T_\A[-1]\rangle$.
 \end{defn}
	
\subparagraph	 \label{para:Bkernelobjects} Let us set $\B_{\kernel (Z'_0)} = \{ E \in \B : Z_0'(E)=0\}$.  For any $E \in \B_{\kernel (Z'_0)}$, we can fit it in a short exact sequence $0 \to E' \to E \to E'' \to 0$ in $\B$ where $E' \in \F_\A$ and $E'' \in \T_\A [-1]$.  Since $Z'_0(E'')=0$, it follows that $Z'_0(E')=0$ and hence $E' \in \A_{\kernel (Z'_0)}$.  Moreover, since $\Hom (\T_\A, E')=0$, it follows from Lemma \ref{Lem:Kertorpair} that $E'$ is a direct sum of copies of $L_0[1]$ while $E''$ is a direct sum of copies of $L_1$.  Now, note that $\Ext^1 (L_1, L_0[1]) \cong \Hom (L_0, L_1) \cong H^0(X, \OO_X (\Theta -f))$, which vanishes because $\OO_X (\Theta -f)$ is a line bundle of negative degree with respect to an ample divisor of the form $\Theta + (2+\epsilon)f$ for a small $\epsilon >0$.  Overall, we see that $E$ itself  is isomorphic to a direct sum of copies of $L_0[1]$ and $L_1$.

   We define $\B_{\ker(Z'_0)}:=\B\cap \ker(Z'_0)$. 
   Consider the $\mathbb{P}^1$ whose coordinate is given by $[D_{\omega}: V_{\omega}]$. Let $a\in \mathbb{P}^1$.  For any $K\in \B_{\ker(Z'_0)}$, we define 
 \[
 \phi_a^R(K)=\lim\limits_{V_\omega, D_\omega\to 0, [V_\omega:D_\omega]=a}\phi_{\sigma'}(K),
 \]
 where $\sigma'=(Z^{td}_{\omega', B'}, \Coh^{\omega', B'})$, and $Z^{td}_{\omega', B'}$ is considered as a function of $V_\omega$, $D_\omega$ via precomposing with $\Phi_Z$.
 
 Consider the triple $$\sigma^R_a=(Z'_0, \B, \{\phi^R_a(K)\}_{K\in \B_{\ker(Z'_0)}}).$$ 
	\begin{lem}
		\label{Lem:Weakseesaw}
		The triple $\sigma^R_a$ satisfies condition (i) and (ii) in Definition \ref{Def:Weak}.
	\end{lem}
	\begin{proof}

 Since for any $K \in \B_{\ker (Z_0')}$ we have
 \begin{equation*}
		\lim\limits_{V_\omega, D_\omega\to 0, [V_\omega:D_\omega]=a}\phi_{\sigma'}(L_1)=\frac{1}{4}\leq\lim\limits_{V_\omega, D_\omega\to 0, [V_\omega:D_\omega]=a}\phi_{\sigma'}(K)\leq\lim\limits_{V_\omega, D_\omega\to 0, [V_\omega:D_\omega]=a}\phi_{\sigma'}(L_0[1])=\frac{3}{4},
	\end{equation*}
 we see that condition (i) is satisfied. 
 The proof of condition (ii) is the same as Proposition \ref{Prop:weakseesawbefore}.
	\end{proof}
	
	To show $\sigma_a^R$ defines a weak stability condition, we are left with checking the HN property.

	\begin{prop}
		\label{Prop:NoeCat}
		The abelian category $\B$ is a noetherian abelian category. 
	\end{prop}
	\begin{proof}
	Since $\omega'$ is ample, it is well known that $\A$ is noetherian, e.g. see \cite[Lemma 6.17]{MSlec}. We use the criterion of \cite[Lemma 5.5.2]{BVdB03} to show $\B$ is noetherian. 
		In particular, we show that for any accending chain 
		\begin{equation}
			\label{Equ:Chain}
			F_0\subset F_1\subset\cdots
		\end{equation}
		with $F_i\in \F_\A$ and $\mathrm{coker}(F_0\rightarrow F_i)\in \T_\A$ stabilizes after a finite number of steps.
		
		Consider the short exact sequence in $\A$: 
		$$0\rightarrow F_0\rightarrow F_i\rightarrow G_i\rightarrow 0.$$
		By the snake lemma, we have 
		\begin{center}
		\begin{tikzcd}
			& & 0\ar{d}& 0\ar{d}\\
			0	\ar{r} & F_0 \ar[equal]{d}\ar{r} &F_i \ar{d}\ar{r} & G_i\ar{d}\ar{r} &0\\
			0	\ar{r} & F_0 \ar{r} &F_{i+1} \ar{d}\ar{r} & G_{i+1}\ar{d}\ar{r} &0\\
			& & C_i \ar{d}\ar{r}{\simeq} & K_i\ar{d}\\
			& & 0 & 0
		\end{tikzcd}
		\end{center}
		where all the short exact sequences are in $\A$. Since $G_i$ and $G_{i+1}$ are in $\T_\A$, we have $C_i\simeq K_i\in \T_\A$. 
		
		Let us apply $\Hom(L_1, \_)$ to the short exact sequence 
		$$0\rightarrow F_i\rightarrow F_{i+1}\rightarrow C_i\rightarrow 0.$$
 Since $L_1, F_{i+1} \in \B$, we have $0= \Ext^{-1}(L_1, F_{i+1})=\Hom(L_1[1], F_{i+1})$. Since $\Hom(L_1, L_1[1])=H^1(X, \cO_X)=0$, we also have $\Hom(L_1, C_i)=0$. Hence the dimension of $\Hom(L_1, F_i)$ decreases as $i$ increases.  At the point where the dimension of $\Hom (L_1, F_i)$ becomes stationary, we obtain $\Ext^{-1}(L_1,C_i)=0$; since $C_i$ is an extension of copies of $L_1[1]$ which has no self-extensions, it follows that $C_i=0$.  Thus the chain \eqref{Equ:Chain} stabilizes after a finite number of steps.
	\end{proof}

	\begin{thm}
		\label{Thm:HNfilweak}
		Given $a\in \mathbb{P}^1$, the weak stability function $\sigma^R_a$ satisfies the HN property. 
	\end{thm}
	\begin{proof}
		The argument follows the proof of  \cite[Lemma 2.18]{piyaratne2015moduli}.
		
		Since $\B$ is noetherian, we only need to check (i) in the Proposition \ref{Prop:HNfil}. Assuming we have such a chain of subobjects in $\B$, consider the short exact sequence in $\B$:
		$$0\rightarrow E_{i+1}\rightarrow E_i\rightarrow F_i\rightarrow 0.$$
		Since $\Im(Z'_0)\geq 0$ on $\B$, we have $\Im(Z'_0(E_{i+1}))\leq \Im(Z'_0(E_i))$. 
		Since $Z'_0$ has coefficients in $\mathbb{Q}$, we know that the image of $Z'_0$ is discrete. Hence $\Im(Z'_0(E_i))$ is constant for $i\gg 0$. Then $\Im(Z'_0(F_i))=0$ for $i>n$ for some $n$. 
		
		Let $i>n$, if $F_i\notin \B_{\ker(Z'_0)}$, then $\phi(F_i)=1$, contradicting with $\phi(E_{i+1})>\phi(F_i)$.
		
		We must have $F_i\in \B_{\ker(Z'_0)}$.  From \ref{para:Bkernelobjects}, we know $F_i$ must be a direct sum of copies of $L_0[1]$ and $L_1$.  From the Ext-group calculations in the proof of Lemma \ref{Lem:Kertorpair}, we know that $\Ext^1 (F_i, L_1)=0=\Ext^1(F_i,L_0[1])$.     Hence  either the dimension of $\text{Hom}(E_{i}, L_1)$ or $\text{Hom}(E_{i}, L_0[1])$ decreases as $i$ increases, we much have $F_i=0$ for $i$ large enough. 
	\end{proof}
	\bigskip
	
	Intuitively, the local picture near the weak stability conditions $\sigma^L_b$'s and $\sigma^R_a$'s are $\Bl_{(0, 0)}\mathbb{R}^2$.

\newpage

\appendix

\section{Solving a central charge equation} \label{sec:app-solvingcce}

In this section, for given $\mathbb{R}$-divisors $\omega, B$ on a Weirstra{\ss} elliptic K3 surface $X$, we solve the central charge equation
 \begin{equation}\label{eq:cce-var1-copy-2}
Z_{\omega', B'}^{td}(\Phi (-)) = T Z_{\omega, B}(-)
 \end{equation}
for some $T \in \mathrm{GL}^+\!(2,\mathbb{R})$ and $\omega', B'$. Since $\sqrt{\mathrm{td}(X)}=(1,0,1)$ on a K3 surface, 
\begin{align*}
Z^{td}_{\omega, B}(F) &= -\int_X e^{-(B+i\omega)}\ch(F) \sqrt{\mathrm{td}(X)} \\
&= -\ch_2^B (F) + (\tfrac{\omega^2}{2}-1)\ch_0^B(F) + i\omega \ch_1^B(F) \\
&= -\ch_2^B (F) + (V_\omega -1)\ch_0^B(F) + i\omega \ch_1^B(F)
\end{align*}
which differs from $Z_{\omega, B}(F)$ by a change a variable in the coefficient for $\ch_0^B$.  Therefore, to solve the equation \eqref{eq:cce-var1-copy-2}, we first solve the equation 
 \begin{equation}\label{eq:cce-original-2}
Z_{\omega', B'}(\Phi (-)) = T Z_{\omega, B}(-).
 \end{equation}
Even though \eqref{eq:cce-original-2} was solved in \cite[Section 10]{Lo20} (see also \cite[Appendix A]{LM2}), the solution there was written in terms of a coordinate system different from the RDV coordinates introduced in \ref{para:RDVcoord}.  Instead of applying a change of variables to the solution in \cite{Lo20} (which is cumbersome), we outline a direct solution which uses the RDV coordinates.

To begin with, note that 
\begin{align*}
  Z_{\omega, B} &= -\ch_2^B + \tfrac{\omega^2}{2}\ch_0^B + i \omega \ch_1^B \\
  &= -\ch_2^B + V_\omega \ch_0^B + i R_\omega (\Theta + (D_\omega + e)f)\ch_1^B \\
  &= -\ch_2 + R_B\Theta \ch_1 + R_B (D_B+e)f\ch_1 + (V_\omega - V_B)\ch_0 \\
  &\text{\hspace{3cm}} + iR_\omega (\Theta \ch_1 + (D_\omega + e)f\ch_1 - R_B (D_B + D_\omega + e)\ch_0).
\end{align*}
Now we set 
\begin{align}
  Z'_{\omega,B} &= \begin{pmatrix}  1 & -R_B/R_\omega \\
  0 & 1/R_\omega \end{pmatrix} Z_{\omega, B} \label{eq:ZpZtransform}\\
  &= -\ch_2 + L_{\omega, B}f\ch_1 + M_{\omega, B}\ch_0 + i (\Theta \ch_1 + (D_\omega + e)f\ch_1 + N_{\omega, B}\ch_0 ) \notag 
\end{align}
where the terms
\begin{align*}
  L_{\omega, B} &= R_B(D_B-D_\omega) \\
  M_{\omega, B} &= V_\omega - V_B + R_B^2(D_B + D_\omega +e) \\
  N_{\omega, B} &= -R_B (D_B + D_\omega + e)
 \end{align*}
 depend only on $\omega$ and $B$.

Solving \eqref{eq:cce-original-2} is now in turn equivalent to solving
 \begin{equation}\label{eq:cce-var3}
Z'_{\omega', B'}(\Phi (E)) = T Z'_{\omega, B}(E) \text{\quad for all $E \in D^b(X)$}.
 \end{equation}

For an object $E \in D^b(X)$, if we write 
\[
 \ch_0(E)=n, \,\,\,  f\ch_1(E)=d, \,\,\,  \Theta \ch_1(E)=c, \,\,\, \ch_2(E)=s,
\]
then from \cite[6.2.6]{FMNT} we have
\[
\ch_0(\Phi (E)=d, \,\,\, f\ch_1(\Phi E) = -n, \,\,\, \Theta \ch_1(\Phi E)=s-\tfrac{e}{2}d+en, \,\,\, \ch_2(\Phi E) = -c -ed + \tfrac{e}{2}n.
\]
Using this notation for Chern classes, we now have 
\[
Z'_{\omega, B}(E) = -s + L_{\omega, B} d + M_{\omega, B}n + i (c+ (D_\omega +e)d+N_{\omega, B} n)
\]
while
\[
Z'_{\omega', B'}(\Phi E) = -(\tfrac{e}{2}+L_{\omega',B'})n + (M_{\omega',B'}+e)d + c + i(-D_{\omega'}n + (N_{\omega',B'}-\tfrac{e}{2})d+s).
\]
It is then clear that if the following relations are satisfied
\begin{equation}\label{eq:LMN-relations}
-L_{\omega, B}=N_{\omega',B'}-\tfrac{e}{2}, \,\,\, M_{\omega,B}=D_{\omega'}, \,\,\, D_\omega + e = M_{\omega',B'}+e, \,\,\, N_{\omega, B} = -(\tfrac{e}{2}+L_{\omega',B'})
\end{equation}
and we take $T=-i=\begin{pmatrix} 0 & 1 \\ -1 & 0 \end{pmatrix}$, then the equation \eqref{eq:cce-var3} would hold.

The four relations in  \eqref{eq:LMN-relations}, when written out  in RDV coordinates, correspond to
\begin{align}
  R_B(D_B-D_\omega) &= R_{B'}(D_{B'}+D_{\omega'}+e)+\tfrac{e}{2} \label{eq:AG52-42-1}  \\
  V_\omega - V_B + R_B^2(D_B+D_\omega + e) &= D_{\omega'} \label{eq:AG52-42-2} \\
  D_\omega + e &= V_{\omega'}-V_{B'}+R_{B'}^2 (D_{B'}+D_{\omega'}+e)+e \label{eq:AG52-42-3} \\
  -R_B (D_B+D_\omega + e) &= -(\tfrac{e}{2}+R_{B'}(D_{B'}-D_{\omega'})) \label{eq:AG52-42-4} 
\end{align}
respectively.  Noting that for any divisor $W$ we have $V_W = \tfrac{1}{2}R_W^2(2D_W+e)$, relation \eqref{eq:AG52-42-2} gives  
\begin{equation}\label{eq:RDVafter-1}
    D_{\omega'}= V_\omega + R_B^2(D_\omega + \tfrac{e}{2}).
\end{equation}
Adding \eqref{eq:AG52-42-1} and \eqref{eq:AG52-42-4}  together  gives 
\begin{equation}\label{eq:RDVafter-2}
    R_{B'}=\frac{R_B(-2D_\omega -e)}{2D_{\omega'}+e}.
\end{equation}
Now \eqref{eq:AG52-42-3}  yields
\begin{equation}\label{eq:RDVafter-3}
  V_{\omega'}=D_\omega - \frac{R_B^2(2D_\omega + e)^2}{2(2D_{\omega'}+e)}.
\end{equation}
Also, subtracting \eqref{eq:AG52-42-4}  from \eqref{eq:AG52-42-1}  (and noting $e=2$) gives
\begin{equation}\label{eq:RDVafter-4}
    R_{B'}D_{B'}=R_BD_B+R_B-R_{B'}-1.
\end{equation}

Finally, we see that when we the RDV coordinates of $\omega, B, \omega', B'$ satisfy \eqref{eq:RDVafter-1}, \eqref{eq:RDVafter-2}, \eqref{eq:RDVafter-4}, and 
\begin{equation}\label{eq:RDVafter-3-changeofvar}
    V_{\omega'}=D_\omega - \frac{R_B^2(2D_\omega + e)^2}{2(2D_{\omega'}+e)}+1
\end{equation}
(instead of \eqref{eq:RDVafter-3}), and we choose 
\[
T=\begin{pmatrix} 1 & -R_{B'}/R_{\omega'} \\ 0 & 1/R_{\omega'} \end{pmatrix}^{-1}\begin{pmatrix} 0 & 1 \\ -1 & 0\end{pmatrix} \begin{pmatrix} 1 & -R_B/R_\omega \\ 0 & 1/R_\omega \end{pmatrix},
\]
the parameters $B, \omega, B', \omega'$ and $T$ together solve the central charge equation \eqref{eq:cce-var1-copy-2}.

The equations in \eqref{Equ:RDVNoTd_Td} are precisely \eqref{eq:RDVafter-1}, \eqref{eq:RDVafter-2}, \eqref{eq:RDVafter-4}, and \eqref{eq:RDVafter-3-changeofvar}.

\bibliography{refs}{}
\bibliographystyle{plain}

\end{document}